\documentclass[12pt,reqno]{amsart}
\usepackage{graphicx, epsfig}
\usepackage{amsfonts,amsmath,amssymb,cite}
\usepackage{multirow}
\usepackage{color}
\usepackage{theorem}

 \advance \topmargin by -\headheight
 \advance \topmargin by -\headsep
 \evensidemargin \oddsidemargin
\newcounter{theorem}
\def\openthm#1#2{\refstepcounter{theorem}\bigskip

{\noindent\bf#1~\thetheorem\if#2!{. }\else{ (#2).}\fi}
\it}

\def\thmskip{}
\newenvironment{theorem}[1][!]{\openthm{Theorem}{#1}}{\thmskip}
\newenvironment{lemma}[1][!]{\openthm{Lemma}{#1}}{\thmskip}

\newenvironment{proposition}[1][!]{\openthm{Proposition}{#1}}{\thmskip}
\newenvironment{conjecture}[1][!]{\openthm{Conjecture}{#1}}{\thmskip}

\newcounter{remark}
\def\openrem#1#2{\refstepcounter{remark}\bigskip
{\noindent \it \bfseries#1~\theremark\if#2!{. }\else{ (#2). }\fi}}
\newenvironment{remark}[1][!]{\openrem{Remark}{#1}}{\qed}

\def\R{\mathbb{R}}
\def\N{\mathbb{N}}

\def\<{\langle}
\def\>{\rangle}

\def\eps{\varepsilon}
\def\argmin{{\rm argmin}}

\def\E{\mathcal{E}}
\def\P{\mathcal{P}}

\def\Fs{\mathcal{F}}

\def\Ys{\mathcal{Y}}
\def\Us{\mathcal{U}}
\def\As{\mathcal{A}}
\def\Cs{\mathcal{C}}
\def\Fs{\mathcal{F}}
\def\Ks{\mathcal{K}}

\def\ub{u}
\def\vb{v}

\def\yb{y}

\def\Lat{L^{{\rm a}}_F}
\def\Lqcl{L^{{\rm qcl}}_F}
\def\Lqcf{L^{{\rm qcf}}_F}
\def\Lqnl{L^{{\rm qnl}}_F}

\def\Lqnlconj{\widetilde{L}^{{\rm qnl}}_F}
\def\Lqnlconjf{\widehat{L}^{{\rm qnl}}_F}

\def\Lqce{L^{{\rm qce}}_F}
\def\Li{{L}}

\def\Eqcl{\E^{{\rm qcl}}}
\def\Eqnl{\E^{{\rm qnl}}}
\def\Eqce{\E^{{\rm qce}}}
\def\Ea{\E^{{\rm a}}}

\def\Fcrit{F_*}
\def\Fqcl{\Fs^{{\rm qcl}}}
\def\Fqcf{\Fs^{{\rm qcf}}}
\def\Fat{\Fs^{{\rm a}}}

\def\Vqcf{{ \tilde V}}

\def\Wqcf{{ \widetilde W}}

\def\lqnl{\lambda^{{\rm qnl}}}

\def\yat{y^{{\rm a}}}
\def\yqcl{y^{{\rm qcl}}}
\def\yqcf{y^{{\rm qcf}}}

\def\uat{u^{{\rm a}}}
\def\uqcl{u^{{\rm qcl}}}
\def\uqcf{u^{{\rm qcf}}}
\def\uqnl{u^{{\rm qnl}}}
\def\uqce{u^{{\rm qce}}}

\def\alphaoptqcl2{\alpha^{\rm qcl,2,{\infty}}_{\rm opt}}
\def\alphamaxqcl2{\alpha^{\rm qcl,2,{\infty}}_{\rm max}}

\DeclareMathOperator{\cond}{cond}

\newcommand{\smfrac}[2]{{\textstyle \frac{#1}{#2}}}
\newcommand{\lpnorm}[2]{\left\|#1\right\|_{\ell^{#2}_\eps}}

\newcommand{\alert}[1]{{#1}}

\begin{document}

\title[Iterative Methods for the Force-based QC
Approximation]{Iterative Methods for the Force-based Quasicontinuum Approximation: Analysis of a 1D Model Problem}

\author{M. Dobson}

\author{M. Luskin}

\author{C. Ortner}

\date{\today}

\thanks{M. Dobson: CERMICS - ENPC,
6 et 8 avenue Blaise Pascal,
Cit\'e Descartes - Champs sur Marne,
77455 Marne la Vall\'ee Cedex 2, France, dobsonm@cermics.enpc.fr}
\thanks{M. Luskin (Corresponding Author):
School of Mathematics, 206 Church St. SE,
  University of Minnesota, Minneapolis, MN 55455, USA,
luskin@umn.edu}
\thanks{Christoph Ortner:
Mathematical Institute, St. Giles' 24--29, Oxford OX1 3LB, UK, ortner@maths.ox.ac.uk}

\thanks{ This work was supported in part by DMS-0757355, DMS-0811039,
  the Department of Energy under Award Numbers DE-FG02-05ER25706
and DE-SC0002085, the University of
  Minnesota Supercomputing Institute, the University of Minnesota
  Doctoral Dissertation Fellowship, the NSF Mathematical Sciences
  Postdoctoral Research Fellowship, and the EPSRC critical mass
  programme ``New Frontier in the Mathematics of Solids.''  }


\begin{abstract}
  Force-based atomistic-continuum hybrid methods are the only known
  pointwise consistent methods for coupling a general atomistic model
  to a finite element continuum model. For this reason, and due to
  their algorithmic simplicity, force-based coupling methods have
  become a popular class of atomistic-continuum hybrid models as well
  as other types of multiphysics models. However, the recently
  discovered unusual stability properties of the linearized
  force-based quasicontinuum (QCF) approximation, especially its
  indefiniteness, present a challenge to the development of efficient
  and reliable iterative methods.

  We present analytic and computational results for the generalized
  minimal residual (GMRES) solution of the linearized QCF equilibrium
  equations.  We show that the GMRES method accurately reproduces the
  stability of the force-based approximation and conclude that an
  appropriately preconditioned GMRES method results in a reliable and
  efficient solution method.
  \end{abstract}

\maketitle

\vspace{6mm}

\section{Introduction}
\alert{The motivation for coupled atomistic/continuum models of solids
  is that the accuracy of an atomistic model is often only needed in
  localized regions of the computational domain, but a coarse-grained
  continuum model is necessary for the simulation of large enough
  systems} to include long-range effects ~\cite{LinP:2006a,
  Gunzburger:2008a, arlequin1, blan05, gaviniorbital, Ortner:2008a,
  ted03, Ortiz:1996a,kohlhoff,cadd}.  The force-based approach has
become very popular because \alert{it provides a particularly simple
  and accurate \cite{dobs-qcf2} method} for coupling two physics
models without the development of an accurate hybrid coupling energy.
It operates by creating disjoint subdomains in which the equilibrium
equations at each degree of freedom are obtained by assigning forces
directly from one of the physics models.  In addition to coupling
atomistic and continuum models, such an approach has also been found
to be attractive, for example, in the coupling of regions modeled by
quantum mechanics to regions modeled by molecular mechanics, since
accurate hybrid coupling energies require an interfacial region that
is too computationally demanding for the quantum mechanics
model~\cite{hybrid_review}.

The force-based quasicontinuum (QCF) approximation is attractive
because of its simple and efficient implementation and because it is
the only known pointwise consistent quasicontinuum (QC) approximation
for coupling a general atomistic model with a Cauchy-Born continuum
model~\cite{dobs-qcf2}. \alert{By {\em consistent} we mean that the
  absence of ghost forces under homogeneous deformations.} Its main
drawback is that it results in a non-conservative force
field~\cite{Dobson:2008a}, that is, the QCF forces are not compatible
with any energy functional.  Several creative attempts have been made
to develop hybrid coupling energies that satisfy the patch test (there
are no resultant forces under uniform strain)
~\cite{Shimokawa:2004,E:2006}, which is a weaker compatibility
condition than pointwise consistency and leads to reduced accuracy.

In this paper, we consider the force-based quasicontinuum
approximation (QCF),
\begin{equation}
  \label{fullqcfc2}
  -\Fqcf(\yqcf)=f,
\end{equation}
but, for simplicity, we will focus mainly on its linearization about a
reference state,
\begin{displaymath}
  \Lqcf \uqcf = f;
\end{displaymath}
see Section \ref{sec:model} for the precise definitions.  Recent
analyses of the linearized QCF
operator~\cite{dobs-qcf2,sharpstabilityqcf} have identified both
further advantages as well as disadvantages of the force-based
coupling approach.  In addition to being non-symmetric, which is
related to the fact that $\Fqcf$ is non-conservative, the linearized
QCF operator also suffers from a lack of
positive-definiteness~\cite{dobs-qcf2}.  In the present paper, we show
that this somewhat unusual stability property of the operator $\Lqcf$
presents a challenge for the development of efficient and stable
iterative solution methods that is overcome by the GMRES methods we
propose.

\subsection{Framework for iterative solution methods}
We consider three related approaches to the development of iterative
methods for the QCF equilibrium
equations~\eqref{fullqcfc2}.
A popular approach \cite{Miller:2008} to solve the force-based
equations~\eqref{fullqcfc2}
modifies a nonlinear conjugate gradient algorithm by replacing the
univariate optimization of an energy, used for step size
selection~\cite{NocedalWright99}, with the computation of a step size
such that the residual is (approximately) orthogonal to the current
search direction.  We will show in Section~\ref{sec:modifiedcg} that,
due to the indefiniteness of $\Lqcf$, this method is not numerically
stable for our QCF model problem.

The second approach we consider is the nonlinear splitting
\begin{displaymath}
  -\Fqcf(y)=-\left[\Fqcf(y)+\nabla \E (y)\right] + \nabla \E (y)
\end{displaymath}
to construct the nonlinear iteration equation
\begin{equation}\label{split}
 \nabla \E (y^{(n+1)})=f+ \left[\Fqcf(y^{(n)})+\nabla \E (y^{(n)})\right] .
\end{equation}
The iterative solution of the nonlinear splitting method~\eqref{split} can then
be obtained from the minimization of the sum of $\E(y)$ and the
potential energy of the dead load $f+g^{(n)}$ where
\[
g^{(n)}:=\Fqcf\left(y^{(n)}\right)+\nabla \E \left(y^{(n)}\right),
\]
that is,
\begin{equation*}
y^{(n+1)} \in \argmin \, \big\{ y \mapsto \E(y) - \< f, y
\>- \< g^{(n)}, y \> \big\}.
\end{equation*}
For this approach to be accurate under conditions near the formation
or motion of defects, care must be taken to ensure that the energy
$\E(y)$ accurately reproduces the stability of the approximated
atomistic system.  We will see in Section~\ref{sec:gfc_method} that
using the original quasicontinuum energy $\Eqce(y)$ \alert{defined in
  \eqref{qce.equation}}, which results in the ghost force correction
(GFC) scheme, does not reliably reproduce the stability of the
atomistic system~\cite{doblusort:qce.stab} and can give a reduced
critical strain for a lattice instability.

To develop the final approach,
we recall the Newton method
\begin{equation}\label{newt}
-\nabla \Fqcf(y^{(n)})[y^{(n+1)}-y^{(n)}]=r^{(n)},
\end{equation}
where $r^{(n)}$ is the residual
\begin{displaymath}
  r^{(n)}:=f+\Fqcf(y^{(n)}).
\end{displaymath}
The GMRES methods proposed and analyzed in this paper apply to
the solution of the linear Newton equations \eqref{newt} or their
approximations.  Since the QCF equilibrium equations are generally
solved along a quasi-static process~\cite{dobsonluskin08}, a good
initial guess is usually available and a small number of iterations of
the outer iteration \eqref{newt} is sufficient to maintain stability
and accuracy.

\subsection{Outline}
We begin in Sections \ref{sec:model} and \ref{sec:qcf} by introducing
the most important quasicontinuum approximations and outlining their
stability properties, which are mostly straightforward generalizations
of results from
\cite{dobs-qcf2,sharpstabilityqcf,doblusort:qce.stab}. We also present
careful numerical studies of the spectral properties of $\Lqcf$ which
are particularly useful for the analysis of Krylov subspace methods in
Section \ref{sec:gmres}.

In Section~\ref{sec:nonlinear}, we revisit the {\em ghost force
  correction (GFC) scheme}~\cite{Shenoy:1999a} which, as was pointed
out in~\cite{Dobson:2008a}, can be understood as a linear stationary
iterative method~\eqref{split} for solving the QCF equilibrium
equations.  We show that, even though the QCF method itself is stable
up to a critical strain $F_*$, the GFC scheme becomes unstable at a
significantly reduced strain for our model problem. This leads us to
conclude (though the simple examples we analyze here can only be first
indicators) that the GFC method is not universally reliable near
instabilities.  We note, however, that the GFC method can be expected
on the basis of both theoretical~\cite{doblusort:qce.stab} and
computational results~\cite{doblusort:qce.stab,Miller:2008} to be more
accurate near instabilities than the use of the uncorrected QCE energy
$\Eqce(y),$ as explained in Section~\ref{sec:gfc_method}.  Numerical
results have also shown that the GFC method can give an accurate
approximation of critical loads if the atomistic-to-continuum
interface is sufficiently far from the defect~\cite[Figure
16]{Miller:2008}, at a cost of a larger atomistic region than likely
required by the accuracy of the QCF approximation.

The quasi-nonlocal energy $\Eqnl(y)$ of~\cite{Shimokawa:2004} given
by~\eqref{qnlenergy} is a more reliable and accurate
energy to use in the splitting iteration~\eqref{split}.  It has been shown
to reproduce the atomistic stability of one-dimensional atomistic systems
with next-nearest neighbor interactions~\cite{doblusort:qce.stab}, and the error for
multi-dimensional atomistic systems is likely to be acceptable
if the longer-range interactions decay sufficiently fast.
The splitting iteration~\eqref{split} can then be used as part of
a continuation algorithm for a quasi-static process~\cite{dobsonluskin08}
that provides the reliable
detection of the stability of the atomistic
system~\cite{doblusort:qce.stab} as well as the improved accuracy for
the deformation given by the force-based
approximation~\cite{dobs-qcf2}.

We conclude
Section~\ref{sec:nonlinear} by proving the numerical instability
of the modified conjugate algorithm~\cite{Miller:2008}
for our QCF model problem.
We present these two examples to demonstrate the subtleties in
designing an iterative algorithm for the solution of the QCF system
and to underscore the need for thorough
numerical analysis in the development of stable and efficient
iterative methods for the QCF system.

We conclude by considering in Section \ref{sec:gmres} the
generalized minimal residual method
(GMRES) for the solution of the indefinite and
non-symmetric QCF system.  We provide an analysis of basic as
well as preconditioned GMRES methods. We find in this section that
a non-standard preconditioned GMRES method, based on the discrete
$W^{1,2}$-inner product, appears to have excellent
stability properties up to the critical strain $F_*$
and a more reliable termination criterion.

\section{Quasicontinuum Approximations and Their Stability}
\label{sec:model}
In this section, we give a condensed description of the prototype QC
approximations and their stability properties.  We refer the reader to
\cite{sharpstabilityqcf,doblusort:qce.stab} for more details. Many
details of this section can be skipped on a first reading and only
referred back to when required.

\subsection{Notation}
\label{sec:notation}
Before we introduce the atomistic model and its QC approximations, we
define the notation that will be used throughout the paper.

We consider a one-dimensional atomistic chain whose $2N+1$ atoms have
the reference positions $x_j = j \eps $ for $\eps = 1/N.$ We will
constrain the displacement of boundary atoms which gives rise to the
{\em displacement space}
\begin{displaymath}
  \Us = \big\{ u \in \R^{2N+1} : u_{-N} = u_N = 0 \big\}.
\end{displaymath}
We will equip the space $\Us$ with various norms which are discrete
variants of the usual Sobolev norms that arise naturally in the
analysis of elliptic PDEs. For displacements $v \in \Us$ and $1 \leq p
\leq \infty,$ we define the $\ell^p_\eps$ norms,
\begin{displaymath}
  \lpnorm{v}{p} := \begin{cases}
    \Big( \eps\sum_{\ell=-N+1}^N |v_\ell|^p \Big)^{1/p},
    &1 \leq p < \infty, \\
    \max_{\ell = -N+1,\dots,N} |v_\ell|,
    &p = \infty,
  \end{cases}
\end{displaymath}
and we let $\Us^{0,p}$ denote the space $\Us$ equipped with the
$\ell^p_\eps$ norm. The inner product associated with the
$\ell^2_\eps$ norm is
\begin{equation*}
  \<v,w\> := \eps \sum_{\ell=-N+1}^N v_\ell w_\ell \qquad
  \text{ for } v, w \in \Us.
\end{equation*}
In fact, we use $\|f\|_{\ell^p_\eps}$ and $\<f, g\>$ to denote the
$\ell^p_\eps$-norm and $\ell^2_\eps$-inner product for arbitrary
vectors $f, g$ which need not belong to $\Us$. In particular, we
further define the $\Us^{1,p}$ norm
\begin{equation}\label{equivdef}
  \|v\|_{\Us^{1,p}} := \|v'\|_{\ell^p_\eps},
\end{equation}
where $(v')_\ell = v_\ell'=\eps^{-1}(v_{\ell}-v_{\ell-1})$, $\ell =
-N+1, \dots, N$, and we let $\Us^{1,p}$ denote the space $\Us$
equipped with the $\Us^{1,p}$ norm.  Similarly, we define the space
$\Us^{2,p}$ and its associated $\Us^{2,p}$ norm, based on the centered
second difference $v_\ell'' = \eps^{-2}(v_{\ell+1} - 2 v_\ell +
v_{\ell-1})$ for $\ell=-N+1,\dots,N-1.$ (We remark that, for $v \in
\Us$, we have that $v' \in \R^{2N}$ and $v'' \in \R^{2N-1}$.)


For a linear mapping $A: \Us_1 \to \Us_2$ where $\Us_i$ are vector
spaces equipped with the norms $\|\cdot \|_{\Us_i},$ we denote the
operator norm of $A$
\[
\|A\|_{L(\Us_1,\ \Us_2)}:=\sup_{v\in\Us,\,v\ne 0}\frac {\|Av\|_{\Us_2}}{\|v\|_{\Us_1}}.
\]
If $\Us_1=\Us_2$, then we use the more concise notation
\[
\|A\|_{\Us_1} :=\|A\|_{L(\Us_1,\ \Us_1)}.
\]

If $A: \Us^{0,2} \to \Us^{0,2}$ is invertible, then we can define the
{\em condition number} by
\[
\cond(A)=\|A\|_{\Us^{0,2}}\cdot \|A^{-1}\|_{\Us^{0,2}}.
\]
When $A$ is symmetric and positive definite, we have that
\[
\cond(A)=\lambda_{2N-1}^A/\lambda_1^A
\]
where the eigenvalues of $A$ are
\[
0<\lambda^A_1\le \dots \le \lambda_{2N-1}^A.
\]
If a linear mapping $A: \Us \to \Us$ is symmetric and positive
definite, then we can define the $A$-inner product and $A$-norm by
\[
\<v,w\>_A:=\<Av,w\>,\qquad \|v\|_A^2=\<Av,v\>.
\]

We define the discrete Laplacian $\Li:\Us\to\Us$ by 
\begin{equation}\label{lap}
(\Li \vb)_j:=-v_j''=
	\left[\frac{-v_{j+1} + 2 v_{j} - v_{j-1}}{\eps^2} \right], \quad j=-N+1,\dots, N-1.
\end{equation}
A definition of the $\Us^{1,2}$ inner product and norm that is
equivalent to \eqref{equivdef} can now be given by
\begin{equation}\label{pos}
\<v,w\>_{\Us^{1,2}}:=\<\Li v,w\>,\qquad \|v\|_{\Us^{1,2}}^2=\<\Li v,v\>
=\|\Li^{1/2}v\|_{\ell^2_\eps}^2=\|v'\|_{\ell^2_\eps}^2.
\end{equation}
Since $\Li^{-1}:\Us\to\Us$ is symmetric and positive definite, we can also define the
$\Us^{-1,2}$ inner product and ``negative'' norm by
\begin{equation}\label{neg}
\<v,w\>_{\Us^{-1,2}}:=\<\Li^{-1}v,w\>,\qquad \|v\|_{\Us^{-1,2}}^2=\<\Li^{-1}v,v\>=\|\Li^{-1/2}v\|_{\ell^2_\eps}^2.
\end{equation}

\subsection{The atomistic model}
\label{sec:model_at}
We consider a one-dimensional atomistic chain whose $2N+3$ atoms have
the reference positions $x_j = j\eps $ for $\eps = 1/N,$ and interact
only with their nearest and next-nearest neighbors. (For an
explanation why we require $2N+3$ instead of $2N+1$ atoms as
previously stated, we note that the atoms with indices $\pm(N+1)$ will
later be removed from the model, and refer to Remark \ref{rem:abcs}
for further details.) We denote the deformed positions by $y_j$,
$j=-N-1,\dots,N+1;$ and we constrain the boundary atoms and their
next-nearest neighbors to match the uniformly deformed state, $y_j^F =
Fj\eps ,$ where $F>0$ is a macroscopic strain,
that is,
\begin{equation}\label{bc}
\begin{aligned}
  y_{-N-1}&=-F(N+1)\eps,\qquad &y_{-N}&=-F N \eps,\\
y_{N}&=FN\eps,\qquad &y_{N+1}&=F(N+1)\eps.
\end{aligned}
\end{equation}
The total energy
of a deformation ${\yb} \in \mathbb{R}^{2N+3}$ is given by
\begin{equation*}
  \Ea(\yb) -\sum_{j=-N}^{N}\eps f_j y_j,
\end{equation*}
where
\begin{equation}\label{atomaa}
\begin{split}
\Ea(\yb)&:=
\sum_{j=-N}^{N+1} \eps \phi\Big(\frac{y_{j} - y_{j-1}}{\eps} \Big)
+ \sum_{j = -N+1}^{N+1} \eps \phi\Big(\frac{y_{j} - y_{j-2}}{\eps}\Big)\\
&=\sum_{j=-N}^{N+1} \eps \phi(y_j')
+ \sum_{j = -N+1}^{N+1} \eps \phi(y_j'+y_{j-1}').
\end{split}
\end{equation}
Here, $\phi$ is a scaled two-body interatomic potential (for example,
the normalized Lennard-Jones potential, $\phi(r) = r^{-12}-2 r^{-6}$),
and $f_j$, $j = -N, \dots, N,$ are external forces. We do not apply a
force at the atoms $\pm(N+1)$, which will later be removed from the
model. The equilibrium equations are given by the force balance
conditions at the unconstrained atoms,
\begin{equation}
\label{full}
\begin{aligned}
-\Fat_j(\yat) &= f_j&\quad&\text{for} \quad j = -N+1, \dots, N-1,\\
\yat_{j}&=Fj\eps&\quad&\text{for} \quad j = -N-1,\,-N,\,N,\,N+1,
\end{aligned}
\end{equation}
where the atomistic force (per lattice spacing $\eps$) is given by
\begin{equation} \label{atomforce}
\begin{split}
\Fat_j(y)
&:=-\frac{1}{\eps}\frac{\partial \E^a(\yb)}{\partial y_j} \\
&=\frac{1}{\eps}\Big\{\left[\phi'(y_{j+1}')
+ \phi'(y_{j+2}'+y_{j+1}')\right]
-\left[\phi'(y_j')
+ \phi'(y_j'+y_{j-1}')\right]\Big\}.
\end{split}
\end{equation}

\subsubsection{Linearization of $\Fat$.}
To linearize \eqref{atomforce} we let $u \in \R^{2N+3}$, $u_{\pm N} =
u_{\pm (N+1)} = 0$, \alert{be a displacement from the homogeneous
  state} $y_j^F = Fj\eps;$ that is, we define
\begin{align*}
  & u_j = y_j - y_j^F  \quad \text{ for } j = -N-1,\dots,N+1.
\end{align*}
We then linearize the atomistic equilibrium equations \eqref{full}
about the homogeneous state $\yb^F$ and obtain a linear system for the
displacement $\ub^a$,
\begin{equation*}
\begin{aligned}
(\Lat \uat)_j&=f_j&\quad&\text{for} \quad j = -N+1, \dots, N-1,\\
u^{a}_{j}&=0&\quad&\text{for} \quad j = -N-1,\,-N,\,N,\,N+1,
\end{aligned}
\end{equation*}
where $(\Lat\vb)_j$ is
given by
\begin{equation*}
(\Lat \vb)_j:=\phi''_F
      \left[\frac{-v_{j+1} + 2 v_{j} - v_{j-1}}{\eps^2} \right]
      + \phi''_{2F}
      \left[\frac{-v_{j+2} + 2 v_{j} - v_{j-2}}{\eps^2} \right].
\end{equation*}
Here and throughout we define
\[
\phi''_{F} := \phi''(F)\quad\text{and}\quad
\phi''_{2F} := \phi''(2F),
\]
where $\phi$ is the interatomic potential in~\eqref{atomaa}.  We will
always assume that $\phi''_F > 0$ and $\phi''_{2F} < 0,$ which holds
for typical pair potentials such as the Lennard-Jones potential under
physically realistic strains $F$. \alert{For example, if $\phi$ is the
  Lennard--Jones potential, and if $\phi''(r_t) = 0$ then
  $\phi'(r_t/2) / \phi(r_t) \approx 1.2 \times 10^4$. This shows that
  the force to compress a chain to achieve a strain $F$ for which
  $\phi''(2F) < 0$ is several orders of magnitude larger than the
  force to fracture the chain.}

\subsubsection{Stability of $\Lat$. }
The stability properties of $\Lat$ can be best understood by using a
representation derived in \cite{doblusort:qce.stab},
\begin{equation}
  \label{eq:Eq_second_diff_form}
  \begin{split}
    \< \Lat u, u \>
    &= \eps A_F \sum_{\ell = -N+1}^N |u_\ell'|^2
    - \eps^3 \phi_{2F}'' \sum_{\ell = -N}^{N} |u_\ell''|^2
    = A_F \|u'\|_{\ell^2_\eps}^2 - \eps^2 \phi_{2F}'' \|u''\|_{\ell^2_\eps}^2, \\
  \end{split}
\end{equation}
where $A_F$ is the continuum elastic modulus
\[
A_F=\phi_{F}'' + 4 \phi_{2F}''.
\]
Following the argument in \cite[Prop. 1]{doblusort:qce.stab}, we prove the following equality
 in ~\cite{qcf.iterative} which describes the stability of the uniformly stretched
 chain.

\begin{proposition}
  If $\phi_{2F}'' \leq 0$, then
  \begin{displaymath}
    \min_{ \substack{u \in \R^{2N+3} \setminus\{0\} \\ u_{\pm N} = u_{\pm (N+1)} = 0 } }
    \frac{\< \Lat u, u \>}{\|u'\|_{\ell^2_\eps}^2} = A_F - \eps^2 \nu_{\eps} \phi_{2F}'',
  \end{displaymath}
  where $0 < \nu_{\eps} \leq C$ for some universal constant
  $C$.
\end{proposition}

\subsubsection{The critical strain} The previous result shows, in
particular, that $\Lat$ is positive definite, uniformly as $N \to
\infty$, if and only if $A_F > 0$. For realistic interaction
potentials, $\Lat$ is positive definite in a ground state $F_0 >
0$. For simplicity, we assume that $F_0 = 1$, and we ask how far the
system can be ``stretched'' by applying increasing macroscopic strains
$F$ until it loses its stability. In the limit as $N \to \infty$, this
happens at the {\em critical strain} $\Fcrit$ which solves the
equation
\begin{equation}\label{critical}
  A_{\Fcrit} = \phi''(\Fcrit) + 4 \phi''(2 \Fcrit) = 0.
\end{equation}

\begin{remark}
  \label{rem:abcs}
  We introduced the two additional atoms with indices $\pm (N+1)$ so
  that the uniform deformation $y = y^F$ is an equilibrium of the
  atomistic model. As a matter of fact, our choice of boundary
  condition here is very close in spirit to the idea of ``artificial
  boundary conditions'' (see \cite[Section 2.1]{dobs-qcf2} or
  \cite{Knap:2001a}), which are normally used to approximate the
  effect of a far field. In the quasicontinuum approximations that we
  present next, these additional boundary atoms are not required.
\end{remark}

\subsection{The Local QC approximation (QCL)}

The local quasicontinuum (QCL) approximation uses the Cauchy-Born
approximation to approximate the nonlocal atomistic model by a local
continuum model~\cite{Dobson:2008a,Miller:2003a,Ortiz:1996a}.
In our context, the Cauchy-Born approximation reads
\begin{displaymath}
  \phi\left(\eps^{-1}(y_{\ell+1}-y_{\ell-1})\right) \approx
  \smfrac12 \big[ \phi(2y_\ell') + \phi(2y_{\ell+1}')],
\end{displaymath}
and results in the QCL energy, for $y\in\R^{2N+3}$ satisfying the
boundary conditions~\eqref{bc},
\begin{equation}\label{lqcen}
\begin{split}
\Eqcl(\yb)&= \sum_{j = -N+1}^N \eps \left[\phi(y_j')
+ \phi(2y_j') \right] \\
&\qquad+\eps \left[\phi(y'_{-N})+\frac 12 \phi(2y'_{-N})+\phi(y'_{N+1})
+\frac 12 \phi(2y'_{N+1})\right]\\
&= \sum_{j = -N+1}^N \eps \left[\phi(y_j')
+ \phi(2y_j') \right] +\eps\left[ 2\phi(F)+\phi(2F)\right].
\end{split}
\end{equation}
Imposing the artificial boundary conditions of zero displacement from
the uniformly deformed state, $y_j^F = Fj\eps,$ we obtain the QCL
equilibrium equations
\begin{equation*}
\begin{aligned}
-\Fqcl_j(\yqcl) &= f_j&\quad&\text{for} \quad j = -N+1, \dots, N-1,\\
\yqcl_{j}&=Fj\eps& \quad&\text{for} \quad j = -N,\,N,
\end{aligned}
\end{equation*}
where
\begin{equation} \label{lqcforce}
  \begin{split}
    \Fqcl_{j}(\yb)&:=-\frac{1}{\eps} \frac{\partial \Eqcl(\yb)}{\partial y_j}
    =\frac{1}{\eps}\Big\{\left[\phi'(y_{j+1}')
      + 2\phi'(2y_{j+1}')\right]
    -\left[\phi'(y_j')
      + 2\phi'(2y_j')\right]\Big\}.
\end{split}
\end{equation}
In particular, we see from \eqref{lqcforce} that the QCL equilibrium
equations are well-defined with only a single constraint at each
boundary (see also Remark \ref{rem:abcs}), and we can restrict our
consideration to $y\in\R^{2N+1}$ with the boundary conditions
$y_{-N}=-F$ and $y_N=F$.

Linearizing the QCL equilibrium equations \eqref{lqcforce} about the
uniformly deformed state $y^F$ results in the system
\begin{equation*}
\begin{aligned}
(\Lqcl\uqcl)_j&=f_j&\quad&\text{for} \quad j = -N+1, \dots, N-1,\\
\uqcl_{j}&=0&\quad&\text{for} \quad j = -N,\,N,
\end{aligned}
\end{equation*}
where $(\Lqcl\vb)_j$, for a displacement $\vb \in
\Us,$ is given by
\begin{equation*}
(\Lqcl \vb)_j=
(\phi''_F + 4 \phi''_{2F})
	\left[\frac{-v_{j+1} + 2 v_{j} - v_{j-1}}{\eps^2} \right]=-A_Fv_j'', \quad j=-N+1,\dots, N-1.
\end{equation*}

The increased efficiency of the local QC approximation is obtained
when its equilibrium equations~\eqref{lqcforce} are coarsened by
reducing the degrees of freedom, using piecewise linear interpolation
between a subset of the atoms~\cite{Dobson:2008a,Miller:2003a}.  For
the sake of simplicity of exposition, we do not treat coarsening in
this paper.

We note that
\begin{equation*}
\Lqcl=A_F\Li
\end{equation*}
where $\Li:\Us\to\Us$ is the discrete Laplacian~\eqref{lap}.
Since the QCL operator is simply a scaled discrete Laplace operator,
its stability analysis is straightforward:
\begin{displaymath}
  \< \Lqcl u, u \> = A_F \|u'\|_{\ell^2_\eps}^2 \qquad
  \text{for all } u \in \Us.
\end{displaymath}
In particular, it follows that $\Lqcl$ is stable if and only if $A_F >
0$, that is, if and only if $F < F_*,$ where $F_*$ is the critical
strain defined in~\eqref{critical}.

\subsection{The force-based QC approximation (QCF)}

In order to combine the accuracy of the atomistic model with the
efficiency of the QCL approximation, the force-based quasicontinuum
(QCF) method decomposes the computational reference lattice into an
{\it atomistic region} $\mathcal{A}$ and a {\it continuum region}
$\mathcal{C}$, and assigns forces to atoms according to the region
they are located in. Since the local QC energy~\eqref{lqcen}
approximates $y_j'+y_{j-1}'$ in \eqref{atomaa} by $2y_j',$ it is clear
that the atomistic model should be retained wherever the strains are
varying rapidly. The QCF operator is given
by~\cite{Dobson:2008a,dobsonluskin08}
\begin{equation}
\label{qcfdefF}
\Fqcf_j(\yb)=
\begin{cases}
\Fat_j(\yb)& \text{if $j\in \mathcal{A}$},\\
\Fqcl_j(\yb)& \text{if $j\in \mathcal{C}$},
\end{cases}
\end{equation}
and the QCF equilibrium equations by
\begin{equation*}
\begin{aligned}
-\Fqcf_j(\yqcf) &= f_j&\quad&\text{for} \quad j = -N+1, \dots, N-1,\\
\yqcf_{j}&=Fj\eps& \quad&\text{for} \quad j = -N,\,N.
\end{aligned}
\end{equation*}
We recall that $\Fqcf$ is a non-conservative
force field and cannot be derived from an
energy~\cite{Dobson:2008a}.

For simplicity, we specify the atomistic and continuum regions as
follows. We fix $K \in \N$, $1 \leq K \leq N-2$, and define
\begin{displaymath}
  \mathcal{A} = \{-K,\dots,K\} \quad \text{and} \quad
  \mathcal{C} = \{-N+1, \dots, N-1\} \setminus \mathcal{A}.
\end{displaymath}
Linearization of ~\eqref{qcfdefF} about $\yb^F$ reads
\begin{equation}
  \label{qcf}
  \begin{aligned}
    (\Lqcf \uqcf)_j&=f_j&\quad&\text{for} \quad j = -N+1, \dots, N-1,\\
    \uqcf_{j}&=0&\quad&\text{for} \quad j = -N,\,N,
  \end{aligned}
\end{equation}
where the linearized force-based operator is given explicitly by
\begin{displaymath}
  (\Lqcf v)_j := \left\{
    \begin{array}{ll}
      (\Lqcl v)_j, & \quad\text{for } j \in \Cs, \\
      (\Lat v)_j, & \quad \text{for } j \in \As.
  \end{array} \right.
\end{displaymath}
We note that, since atoms near the artificial boundary belong to
$\mathcal{C}$, only one boundary condition is required at each end.

We know from \cite{sharpstabilityqcf} that the stability analysis of
the QCF operator $\Lqcf$ is highly non-trivial. We will therefore
treat it separately and postpone it to Section \ref{sec:qcf}.


\subsection{The original energy-based QC approximation (QCE)}
In the original energy-based quasicontinuum (QCE) method
\cite{Ortiz:1996a}, an energy functional is defined by assigning
atomistic energy contributions in the atomistic region and continuum
energy contributions in the continuum region. In the context of our
model problem, it can be written as
\begin{equation}\label{qce.equation}
  \Eqce(y) = \eps \sum_{\ell \in \As} \E_\ell^a(y)
  + \eps \sum_{\ell \in \Cs} \E_\ell^c(y) \quad \text{for } y \in \R^{2N+1},
\end{equation}
where
\begin{align*}
  \E_\ell^c(y) =~& \smfrac12 \big( \phi(2y_\ell') + \phi(y_\ell') + \phi(y_{\ell+1}') + \phi(2y_{\ell+1}') \big), \quad \text{and} \\
  \E_\ell^a(y) =~& \smfrac12 \big( \phi(y_{\ell-1}' + y_\ell') + \phi(y_\ell') + \phi(y_{\ell+1}') + \phi(y_{\ell+1}' + y_{\ell+2}') \big).
\end{align*}

The QCE method does not satisfy the patch test~\cite{mingyang,Dobson:2008b,Dobson:2008c,Shenoy:1999a},
which be seen from the existence of
``ghost forces'' at the interface, that is, $\nabla \Eqce(y^F) = g
\neq 0$. Consequently, the linearization of the QCE equilibrium
equations about $y^F$ takes the form~(see \cite[Section
2.4]{Dobson:2008b} and \cite[Section 2.4]{Dobson:2008c} for more
detail)
\begin{equation}
  \label{qce}
  \begin{aligned}
    (\Lqce\uqce)_j-g_j&=f_j&\quad&\text{for} \quad j = -N+1, \dots, N-1,\\
    \uqce_{j}&=0&\quad&\text{for} \quad j = -N,\,N,
  \end{aligned}
\end{equation}
where, for $0 \leq j \leq N-1,$ we have
\begin{equation*}
\begin{split}
(\Lqce \vb)_j &= \phi''_F \frac{-v_{j+1} +2 v_j - v_{j-1}}{\eps^2} \\
& \hspace{-8mm} + \phi_{2F}'' \left\{ \begin{array}{lr}
\displaystyle
 4 \frac{-v_{j+2} +2 v_j - v_{j-2}}{4 \eps^2},
& \hspace{-8mm} 0 \leq j \leq K-2, \\[6pt]
\displaystyle
 4 \frac{-v_{j+2} +2 v_j - v_{j-2}}{4 \eps^2}
+ \frac{1}{\eps} \frac{v_{j+2} - v_{j}}{2 \eps},  & j = K-1, \\[6pt]
\displaystyle
 4 \frac{-v_{j+2} +2 v_j - v_{j-2}}{4 \eps^2}
- \frac{2}{\eps} \frac{v_{j+1} - v_{j}}{\eps}
+ \frac{1}{\eps} \frac{v_{j+2} - v_{j}}{2 \eps},
& j = K, \\[6pt]
\displaystyle
4 \frac{-v_{j+1} +2 v_j - v_{j-1}}{\eps^2}
-  \frac{2}{\eps} \frac{v_{j} - v_{j-1}}{\eps}
+ \frac{1}{\eps} \frac{v_{j} - v_{j-2}}{2 \eps}, & \quad j = K+1, \\[6pt]
\displaystyle
4 \frac{-v_{j+1} +2 v_j - v_{j-1}}{\eps^2}
+ \frac{1}{\eps} \frac{v_{j} - v_{j-2}}{2 \eps}, & j = K+2, \\[6pt]
\displaystyle
4 \frac{-v_{j+1} +2 v_j - v_{j-1}}{\eps^2}, &
\hspace{-16mm} K+3 \leq j \leq N-1,
\end{array}\right.
\end{split}
\end{equation*}
and where the vector of ``ghost forces,'' $g$, is defined by
\begin{equation*}
g_j = \begin{cases}
0, & 0 \leq j \leq K-2, \\
-\frac{1}{2\eps} \phi'_{2F}, & j = K-1, \\
\hphantom{-} \frac{1}{2\eps} \phi'_{2F}, & j = K, \\
\hphantom{-} \frac{1}{2\eps} \phi'_{2F}, & j = K+1, \\
-\frac{1}{2\eps} \phi'_{2F}, & j = K+2, \\
0, &  K+3 \leq j \leq N-1.\\
\end{cases}
\end{equation*}
For space reasons, we only list the entries for $0\le j\le N-1.$ The
equations for $j=-N+1,\dots,-1$ follow from symmetry.

We prove in \cite{qcf.iterative} the following new sharp stability
estimate for the QCE operator $\Lqce$ which implies that the $\Lqce$
operator gives an O($1$) approximation for the critical strain, $F_*.$

\begin{lemma}
  \label{th:stab_qce}
  If $K \geq 1$, $N \geq K+2$, and $\phi_{2F}'' \leq 0$, then
  \begin{displaymath}
    \inf_{\substack{u \in \Us \\ \|u'\|_{\ell^2_\eps} = 1}} \< \Lqce u, u\>
    = A_F + \lambda_K \phi_{2F}'',
  \end{displaymath}
  where $\smfrac12 \leq \lambda_K \leq 1$. Asymptotically, as $K \to
  \infty$, we have
  \begin{displaymath}
    \lambda_K \sim \lambda_* + O(e^{-cK}) \quad
    \text{where } \lambda_* \approx  0.6595
    \text{ and } c \approx 1.5826.
  \end{displaymath}
\end{lemma}

This result will be used in Section~\ref{sec:gfc_method} where we
analyze the ghost-force correction iteration, interpreted as a linear
stationary iterative method for $\Lqcf$ with preconditioner $\Lqce$.

\subsection{The quasi-nonlocal QC approximation (QNL)}
\label{sec:model:qnl}
The QCF method is the simplest idea to circumvent the patch test
failure of the QCE method. An alternative approach was suggested in
\cite{Shimokawa:2004,E:2006}, which is based on a modification of the
energy at the interface. In this model, a next-nearest neighbor
interaction term $\phi(\eps^{-1}(y_{\ell+1}-y_{\ell-1}))$ is left
unchanged if at least one of the atoms $\ell+1, \ell-1$ belong to the
atomistic region or an interface region (which is implicitly defined
by \eqref{qnlenergy}), and is otherwise replaced, preserving symmetry,
by a Cauchy-Born approximation,
\begin{displaymath}
  \phi\left(\eps^{-1}(y_{\ell+1}-y_{\ell-1})\right) \approx
  \smfrac12 \big[ \phi(2y_\ell') + \phi(2y_{\ell+1}')].
\end{displaymath}
This idea leads to the energy functional
\begin{equation}\label{qnlenergy}
  \begin{split}
    \Eqnl(y) := \eps \sum_{\ell = -N+1}^N  \phi(y_\ell')
    + \eps \sum_{\ell \in \As} \phi(y_\ell' + y_{\ell+1}')
    + \eps \sum_{\ell \in \Cs} \smfrac12 \big[ \phi(2y_\ell')
    + \phi(2y_{\ell+1}')\big],
  \end{split}
\end{equation}
where we set $\phi(y_{-N}')=\phi(y_{N+1}')=0$. The QNL approximation
satisfies the patch test; that is, $y = y^F$ is an equilibrium of the QNL energy
functional.

The linearization of the QNL equilibrium equations about the uniform deformation $y^F$ is
\begin{equation*}
  \begin{aligned}
    (\Lqnl \uqnl)_j&=f_j&\quad&\text{for} \quad j = -N+1, \dots, N-1,\\
  \uqnl_{j}&=0&\quad&\text{for} \quad j = -N,\,N,
  \end{aligned}
\end{equation*}
where
\begin{equation}
\label{Lqnl}
\begin{split}
(\Lqnl v)_j &= \phi''_F \frac{-v_{j+1} +2 v_j - v_{j-1}}{\eps^2} \\
& \hspace{-5mm} + \phi_{2F}'' \left\{ \begin{array}{lr}
\displaystyle
 4 \frac{-v_{j+2} +2 v_j - v_{j-2}}{4 \eps^2},
& 0 \leq j \leq K-1, \\[6pt]
\displaystyle
 4 \frac{-v_{j+2} +2 v_j - v_{j-2}}{4 \eps^2}
- \frac{-v_{j+2} + 2 v_{j+1} - v_{j}}{\eps^2},
& j = K, \\[6pt]
\displaystyle
4 \frac{-v_{j+1} +2 v_j - v_{j-1}}{\eps^2}
+ \frac{-v_{j} + 2 v_{j-1} - v_{j-2}}{\eps^2},
& j = K+1, \\[6pt]
\displaystyle
4 \frac{-v_{j+1} +2 v_j - v_{j-1}}{\eps^2}, &
\hspace{-5mm} K+2 \leq j \leq N-1.
\end{array} \right.
\end{split}
\end{equation}
We observe from \eqref{Lqnl} that $\Lqnl$ is not pointwise consistent at
$j=K$ and $j=K+1.$

Repeating our stability analysis for the periodic QNL operator in
\cite[Sec. 3.3]{doblusort:qce.stab} verbatim, we obtain the following
result.

\begin{proposition}
  \label{th:stab_qnl}
  If $K < N-1$, and $\phi_{2F} \leq 0$, then
  \begin{displaymath}
    \inf_{\substack{u \in \Us \\ \|u'\|_{\ell^2_\eps} = 1}} \< \Lqnl u, u \> = A_F.
  \end{displaymath}
\end{proposition}

\begin{remark}\label{depend}
  Since $\phi_{2F}''=(A_F-\phi_F'')/4,$ the linearized operators
  $(\phi''_F)^{-1}\Lat,$ $(\phi''_F)^{-1}\Lqcl,$
  $(\phi''_F)^{-1}\Lqcf,$ $(\phi''_F)^{-1}\Lqce,$ and
  $(\phi''_F)^{-1}\Lqnl$ depend only on $A_F/\phi''_F$, $N$ and $K$.
\end{remark}

\section{Stability and Spectrum of the QCF operator}
\label{sec:qcf}
In this section, we collect various properties of the linearized QCF
operator, which are, for the most part, variants of our results in
\cite{dobs-qcf2,sharpstabilityqcf}. We begin by stating a result
for the lack of positive-definiteness of $\Lqcf,$
which lies at the heart of many of the difficulties one encounters in
analyzing the QCF method.

\begin{theorem}[Lack of Positive-Definiteness of QCF, Theorem 1, \cite{dobs-qcf2}]
  \label{th:nocoerc}
  If $\phi_F'' > 0$ and $\phi_{2F}'' \in \mathbb{R} \setminus \{0\}$
  then, for sufficiently large $N,$ the operator $\Lqcf$ is {\em not}
  positive-definite. More precisely, there exist $N_0 \in \mathbb{N}$
  and $C_1 \geq C_2 > 0$ such that, for all $N \geq N_0$ and $2 \leq K
  \leq N/2$,
  \begin{displaymath}
    - C_1 N^{1/2} \leq  \inf_{\substack{\vb \in \Us \\ \|\vb'\|_{\ell^2_\eps} = 1}}
   \big \langle \Lqcf {\vb}, {\vb} \big\rangle
    \leq - C_2 N^{1/2}.
  \end{displaymath}
\end{theorem}



As a consequence of Theorem \ref{th:nocoerc}, we analyzed the
stability of $\Lqcf$ in alternative norms. Following the proof of
\cite[Theorem 3]{sharpstabilityqcf} verbatim (see also \cite[Remark
3]{sharpstabilityqcf}) gives the following sharp stability result.

\begin{proposition}
  \label{th:sharp_stab_U2inf}
  If $A_F > 0$ and $\phi_{2F}'' \leq 0$, then $\Lqcf$ is invertible with
  \begin{displaymath}
    \big\|(\Lqcf)^{-1}\big\|_{L(\Us^{0,\infty},\ \Us^{2,\infty})} \leq 1/A_F.
  \end{displaymath}
  If $A_F = 0,$ then $\Lqcf$ is singular.
\end{proposition}

\medskip This result shows that $\Lqcf$ is operator stable up to the
critical strain $\Fcrit$ at which the atomistic model loses its
stability as well (cf. Section \ref{sec:model_at}). In the remainder
of this section, we will investigate, in numerical experiments, the
spectral properties of the $\Lqcf$ operator for strains $F$ such that
$A_F > 0$ and $\phi_{2F}'' \leq 0$.

\subsection{Spectral properties of $\Lqcf$ in $\Us^{0,2}=\ell^2_\eps$}
\label{sec:qcf:U02}
The spectral properties of the $\Lqcf$ operator are crucial for
analyzing the performance of iterative methods in Hilbert spaces.  The
basis of our analysis of $\Lqcf$ in the Hilbert space $\Us^{0,2}$ is
the remarkable observation that, even though $\Lqcf$ is non-normal, it
is nevertheless diagonalizable and its spectrum is identical to that
of $\Lqnl$. We first observed this in
\cite[Section~4.4]{sharpstabilityqcf} for the case of periodic
boundary conditions.  Repeating the same numerical experiments for
Dirichlet boundary conditions, we obtain similar results. Table
\ref{tbl:qcf_ana:conj_evals}, where we display the error between the
spectrum of $\Lqcf$ and $\Lqnl,$ gives rise to the following
conjecture.

\begin{table}[t]
  \begin{equation*}
\begin{array}{r|rrrrr}
      &       A_F = 0.8 &       0.6 &       0.4 &       0.2 &      0.04 \\
\hline
     N = 8 &  4.83e\hbox{--}13 &  4.26e\hbox{--}13 &  3.13e\hbox{--}13 &  3.41e\hbox{--}13 &  1.71e\hbox{--}13 \\
    32 &  1.73e\hbox{--}11 &  1.27e\hbox{--}11 &  9.55e\hbox{--}12 &  9.55e\hbox{--}12 &  1.41e\hbox{--}11 \\
   128 &  8.08e\hbox{--}10 &  4.00e\hbox{--}10 &  4.07e\hbox{--}10 &  4.15e\hbox{--}10 &  4.15e\hbox{--}10 \\
   512 &  1.06e\hbox{--}08 &  8.73e\hbox{--}09 &  1.40e\hbox{--}08 &  8.38e\hbox{--}09 &  8.73e\hbox{--}09 \\
\end{array}
  \end{equation*}
  \caption{\label{tbl:qcf_ana:conj_evals} The difference between the
    spectra of $\Lqcf$ and
    $\Lqnl.$ The table displays the $\ell^\infty$ norm of errors in the
    ordered vectors of eigenvalues for various choices of $A_F$ with $\phi_F'' = 1$, for increasing
    $N$, $K = \lfloor\sqrt{N}\rfloor + 1$.  All entries are zero to the
    precision of the eigenvalue solver.  }
  \vspace{-7mm}
\end{table}

\begin{conjecture}
  \label{th:samespec}
  For all $N \geq 4,\ 1 \leq K \leq N-2,\text{ and } F > 0,$ the
  operator $\Lqcf$ is diagonalizable and its spectrum is identical to
  the spectrum of $\Lqnl$.
\end{conjecture}

\medskip
We denote the eigenvalues of $\Lqnl$ (and $\Lqcf$) by
\begin{displaymath}
  0<\lqnl_1\le ... \lqnl_{\ell}\le ...\le \lqnl_{2N-1}.
\end{displaymath}
The following lemma provides a lower bound for $\lqnl_{1},$ an upper
bound for $\lqnl_{2N-1},$ and consequently an upper bound for
$\cond(\Lqnl) = {\lqnl_{2N-1}}/{\lqnl_{1}}$. Assuming the validity of
Conjecture \ref{th:samespec}, this translates directly to a result on
the spectrum of $\Lqcf$.

\begin{lemma}\label{qnlcond}
If $K < N-1$ and $\phi_{2F}'' \leq 0$, then
\begin{equation*}
\begin{gathered}
\lqnl_{1}\ge  2\,A_F, \qquad
\lqnl_{2N-1}\le \left(A_F-4\phi_{2F}''\right)\eps^{-2}=\phi_{F}''\eps^{-2},
\quad \text{and} \\
\cond(\Lqnl) =
\frac{\lqnl_{2N-1}}{\lqnl_{1}} \le
\left(\frac{\phi_{F}''}{2A_F}\right)\eps^{-2}.
\end{gathered}
\end{equation*}
\end{lemma}
\begin{proof}
It follows from Proposition~\ref{th:stab_qnl}
and \eqref{ray1} that
\[
\lqnl_{1}= \inf_{\substack{v \in \Us\\v \neq 0}}
    \frac{\< \Lqnl v,\, v \> }{\<v,\,v\>} =
    \inf_{\substack{v \in \Us\\v \neq 0}}
    \frac{\< \Lqnl v,\, v \> }{\<v',\,v'\>}\cdot \frac{\<v',\,v'\>}{\<v,v\>}
    \ge A_F \inf_{\substack{v \in \Us\\v \neq 0}}\frac{\<v',\,v'\>}{\<v,v\>}
    \ge 2 A_F
\]
since the {\em infimum} of the Rayleigh quotient $\< v', v'\>/\<v,v\>$
is attained for $v \in \Us$ where $v_\ell=\sin((N-\ell)\pi/ (2N))$
~\cite[Exercise 13.9]{SuliMayers} and has the value
\begin{equation}\label{ray1}
\inf_{\substack{v \in \Us\\v \neq 0}}
    \frac{\< v', v'\>}{\<v,v\>}
    =4 N^2 \sin^2\left(  \frac{\pi}{4N} \right) \ge2.
    \end{equation}

The estimate for the maximal eigenvalue follows similarly from
\begin{equation*}
\lqnl_{2N-1} = \sup_{\substack{v \in \Us\\v \neq 0}}
    \frac{\< \Lqnl v, v\>}{\<v,v\>}
\end{equation*}
and the representation \eqref{Lqnl}.
\end{proof}

\medskip For the analysis of iterative methods, particularly the GMRES
method, we are also interested in the condition number of the basis of
eigenvectors of $\Lqcf$ as $N$ tends to infinity. Assuming the
validity of Conjecture \ref{th:samespec}, we can write $\Lqcf = V
\Lambda^{{\rm qcf}} V^{-1}$ where $\Lambda^{{\rm qcf}}$ is
diagonal. In Figure \ref{fig:condV_U02}, we plot the condition number
for increasing values of $N$ and $K$, and for various choices of
$A_F$ with $\phi_{F}''=1$ (it follows from Remark~\ref{depend}
that $V$ actually depends only on $A_F/\phi''_F$ and $N$). Even though it is difficult to determine
from this graph whether ${\rm cond}(V)$ is bounded as $N \to \infty$,
it is fairly clear that the condition number grows significantly
slower than $\log(N)$. We formulate this in the next conjecture.

\begin{conjecture}
  \label{th:qcf:cond_U02}
  Let $V$ denote the matrix of eigenvectors for the force-based QC
  operator $\Lqcf$. If $A_F > 0$, then ${\cond}(V) = o\left(\log(N)\right)$ as $N
  \to \infty$.
\end{conjecture}

\begin{figure}
  \begin{center}
    \includegraphics[width=11cm]{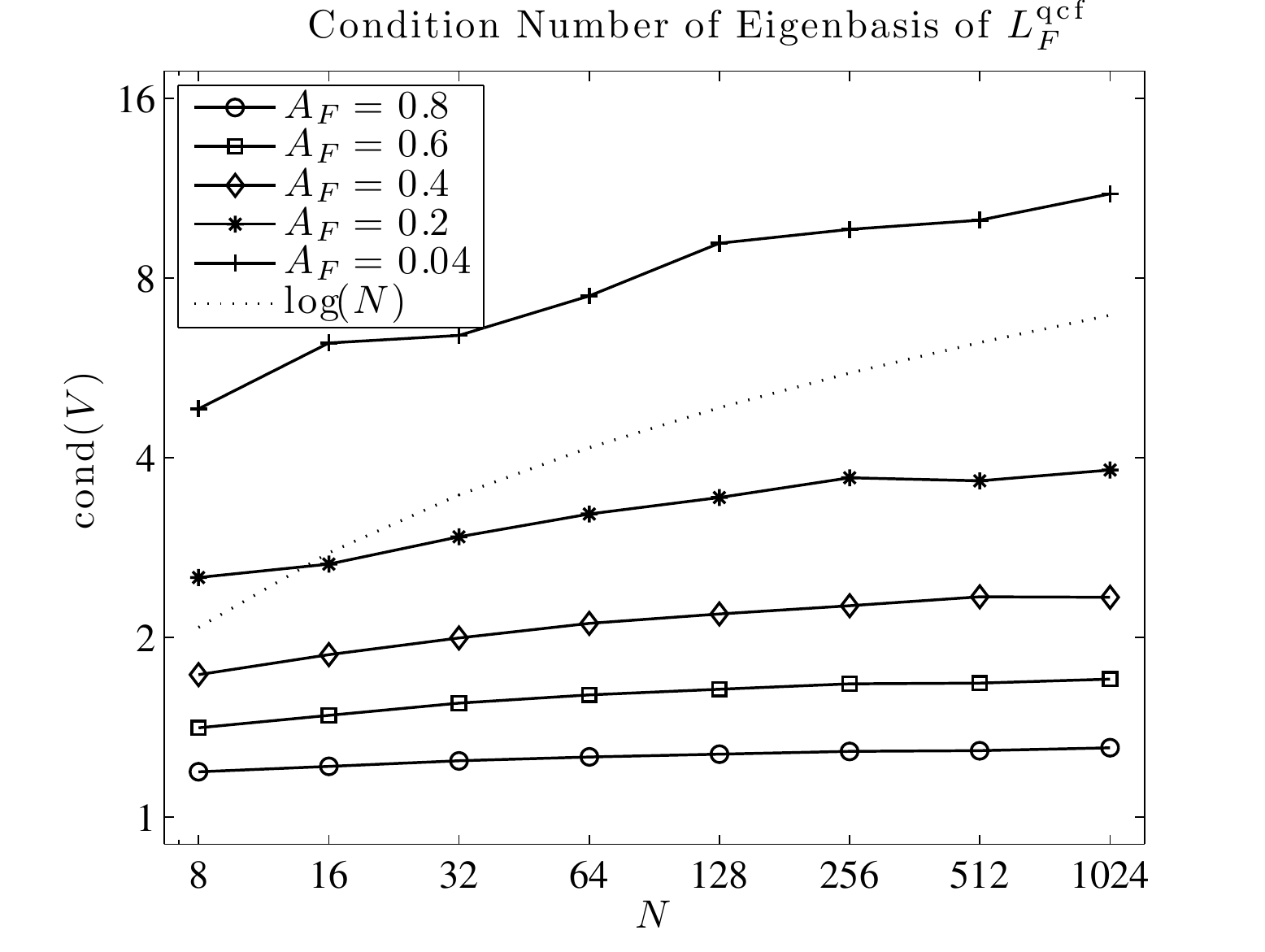}
    \caption{\label{fig:condV_U02} Condition number of the matrix $V$
      plotted against $N$, with atomistic region size $K =
      \lfloor\sqrt{N}\rfloor + 1$, and for various values of $A_F$,
      with fixed $\phi_{F}''=1$. Here, $\Lqcf = V \Lambda^{{\rm qcf}}
      V^{-1}$ is the spectral decomposition of $\Lqcf$. }
  \end{center}
\end{figure}

\medskip

\subsection{Spectral properties of $\Lqcf$ in $\Us^{1,2}$}
\label{sec:qcf:U12}
To study the preconditioning of $\Lqcf$ by
$\Lqcl=A_F\Li,$ we consider
the (generalized) eigenvalue problem
\begin{equation}
  \label{eq:defn_U12-evals}
  \Lqcf v = \lambda \Li v,\qquad v\in\Us,
\end{equation}
which can, equivalently, be written as
\begin{equation}\label{eigv}
  \Li^{-1} \Lqcf v = \lambda v, \qquad v \in \Us,
\end{equation}
or as
\begin{equation}\label{eigw}
  \Li^{-1/2} \Lqcf \Li^{-1/2} w = \lambda w,\qquad w\in\Us,
\end{equation}
with the basis transform $w = \Li^{1/2} v$, in either case reducing it
to a standard eigenvalue problem in $\ell^2_\eps$.

In Table \ref{tbl:qcf:spec_U12}, we display the numerical experiment
that corresponds to the same experiment shown in Table
\ref{tbl:qcf_ana:conj_evals}. We observe that also the
$\Us^{1,2}$-spectra of the $\Lqcf$ and $\Lqnl$ operators are identical
to numerical precision.

\begin{table}[t]
  \begin{equation*}
\begin{array}{r|rrrrr}
     &  A_F =  0.8 &   0.6 &   0.4 & 0.2 &   0.04 \\
\hline
    N =  8 &  3.33e\hbox{--}15 &  1.13e\hbox{--}14 &  1.67e\hbox{--}15 &  2.14e\hbox{--}15 &  9.99e\hbox{--}16 \\
    32 &  1.88e\hbox{--}13 &  1.83e\hbox{--}13 &  4.62e\hbox{--}14 &  6.48e\hbox{--}14 &  3.94e\hbox{--}14 \\
   128 &  1.34e\hbox{--}12 &  5.13e\hbox{--}13 &  5.72e\hbox{--}13 &  3.85e\hbox{--}13 &  5.51e\hbox{--}13 \\
   512 &  2.22e\hbox{--}11 &  9.78e\hbox{--}12 &  7.02e\hbox{--}12 &  4.32e\hbox{--}12 &  4.56e\hbox{--}12 \\
\end{array}
  \end{equation*}
  \caption{\label{tbl:qcf:spec_U12} The difference between the
    spectra of $\Li^{-1} \Lqcf$ and $\Li^{-1} \Lqnl.$ The
    table displays the $\ell^\infty$ norm of errors in the ordered
    vectors of eigenvalues for various choices of $F$, for increasing
    $N$, $K = \lfloor\sqrt{N}\rfloor + 1$, and with fixed $\phi_F'' =
    1$.  All entries are zero to the precision of the eigenvalue
    solver.  }
  \vspace{-7mm}
\end{table}

\begin{conjecture}
  \label{th:spec_U12}
  For all $N \geq 4,\ 1 \leq K \leq N-2,\text{ and } F > 0,$ the
  operator $\Li^{-1}\Lqcf$ is diagonalizable and its spectrum is
  identical to the spectrum of $\Li^{-1}\Lqnl$.
\end{conjecture}

\medskip In the following lemma we completely characterize the
spectrum of $\Li^{-1} \Lqnl$, and thereby, subject to the validity of
Conjecture \ref{th:spec_U12}, also the spectrum of $\Li^{-1}\Lqcf$. We
denote the spectrum of $\Li^{-1}\Lqcf$ by $\{ \mu_j^{\rm qnl} : j = 1,
\dots, 2N-1\}$.

\begin{lemma}
  \label{th:spec_Lqnl_U12}
  Let $K \leq N-2$ and $A_F > 0$, then the (unordered) spectrum of
  $\Li^{-1}\Lqnl$ (that is, the $\Us^{1,2}$-spectrum) is given by
  \begin{align*}
    \mu_j^{\rm qnl} =
    \begin{cases}
      A_F - {4 \phi_{2F}''} \sin^2\big(\smfrac{j \pi}{4 K+4}\big), &
      j = 1, \dots, 2K+1, \\
      A_F, & j = 2K+2, \dots, 2N-1.
    \end{cases}
  \end{align*}
  In particular, if $\phi_{2F}'' \leq 0,$ then
  \begin{displaymath}
    \frac{\max_j \mu_j^{\rm qnl}}{\min_j \mu_j^{\rm qnl}}
    = 1 - \frac{4 \phi_{2F}''}{A_F}
      \sin^2\left(\smfrac{(2K+1)\pi}{4K+4}\right)
      =\frac{\phi_F''}{A_F}+\frac{4 \phi_{2F}''}{A_F}
      \sin^2\left(\smfrac{\pi}{4K+4}\right)
   = \frac{\phi_F''}{A_F} + O(K^{-2}).
  \end{displaymath}
\end{lemma}
\begin{proof}
  We will use the variational representation of $\Lqnl$ from
  \cite[Section 3.3]{doblusort:qce.stab}, which reads
  \begin{displaymath}
    \big\< \Lqnl u, v \big\> = A_F \< u', v' \>
    -  \phi_{2F}'' \eps \sum_{\ell = -K}^{K} (u_{\ell+1}' - u_\ell')
    (v_{\ell+1}' - v_\ell') \quad \text{for } u,\,v \in\Us.
  \end{displaymath}
  Summation by parts in the second term yields
  \begin{displaymath}
    \big\< \Lqnl u, v \big\> = A_F \< u', v'\> - \phi_{2F}'' \< M u', v' \> \quad \text{for } u,\,v \in\Us,
  \end{displaymath}
  where $M$ is the $2N \times 2N$ matrix given by
  \begin{displaymath}
    M = {\scriptsize \left(\begin{array}{rrr|rrrrr|rrr}
      0 &      &   &    &      &      &      &    &   &      &   \\[-1mm]
        &\ddots&   &    &      &      &      &    &   &      &   \\[-1mm]
        &      & 0 &    &      &      &      &    &   &      &   \\
        \hline
        &      &   &  1 &   -1 &      &      &    &   &      &   \\
        &      &   & -1 &    2 &   -1 &      &    &   &      &   \\[-1mm]
        &      &   &    &\ddots&\ddots&\ddots&    &   &      &   \\[-1mm]
        &      &   &    &      &   -1 &    2 & -1 &   &      &   \\
        &      &   &    &      &      &   -1 &  1 &   &      &   \\
        \hline
        &      &   &    &      &      &      &    & 0 &      &   \\[-1mm]
        &      &   &    &      &      &      &    &   &\ddots&   \\[-1mm]
        &      &   &    &      &      &      &    &   &      & 0
    \end{array} \right), }
\end{displaymath}
and where the first and last non-zero rows are, respectively, the rows
$-K$ and $K+1$.  We call the restriction of the {\em conjugate
  operator} $\Lqnlconjf=A_F I - \phi_{2F}'' M:\R^{2N}\to \R^{2N}$ to
the $2N-1$ dimensional invariant gradient space $\R^{2N}_* = \{
\varphi \in \R^{2N} : \sum_\ell \varphi_\ell = 0 \}$ the restricted
conjugate QNL operator $\Lqnlconj=A_F I - \phi_{2F}'' M:\R^{2N}_*\to
\R^{2N}_*,$ and we note that we can write the eigenvalue
relation~\eqref{eq:defn_U12-evals} in weak form as
\begin{equation}\label{same}
    \< \Lqnl u, v \> = \< \Lqnlconj u', v' \> = \lambda \< u',
    v'\> \qquad \forall v\in\Us.
\end{equation}

We can see from \eqref{same} that the $2N-1$
generalized $\Us^{1,2}$-eigenvalues of $\Lqnl$ and the standard
  $\ell^2$-eigenvalues of $\Lqnlconj:\R^{2N}_*\to \R^{2N}_*$ are the same.
  If $\nu_j$ are the $2N-1$ eigenvalues of $\Lqnlconj$ with eigenvectors $\varphi^{(j)}$ in $\R^{2N}_*$; then, letting
  $u^{(j)} \in \Us$ be the (unique) functions for which $(u^{(j)})'
  = \varphi^{(j)}$, we obtain
  \begin{displaymath}
    \big\< \Lqnl u^{(j)}, v \big\> = \big\< \Lqnlconj (u^{(j)})', v' \big\> =
     \nu_j \big\< (u^{(j)})', v'\big\> \qquad \forall v\in\Us,
  \end{displaymath}
  which is equivalent to \eqref{eq:defn_U12-evals}.

  The operator $\Lqnlconj:\R^{2N}_*\to \R^{2N}_*$ has a $(2N -
  2K-2)$-multiple eigenvalue with value $A_F$ and corresponding
  orthogonal eigenvectors $\varphi^{(j)}\in\R^{2N}_*$ can be taken to
  be the projection onto $\R^{2N}_*$ of the canonical basis vectors
  corresponding to the zero-diagonal entries of $M.$
  We will see that the remaining $2K+1$ eigenvalues of $\Lqnlconj:\R^{2N}_*\to
  \R^{2N}_*$ take the form
  \begin{displaymath}
    \nu_j = A_F - \phi_{2F}'' \tilde{\nu}_j,
  \end{displaymath}
  where $\tilde{\nu}_j$, $j = 1, \dots, 2K+1$ are the non-zero
  eigenvalues of the non-zero block of $M$, which we denote
  $\widetilde{M}$. It is easy to check that the eigenvectors of the
  matrix $\widetilde{M}$ are given by
  \begin{displaymath}
    g_\ell^{(j)} = \cos\big( j \pi (\ell + K - 1/2) / (2K+2) \big),
    \quad \ell = -K, \dots, K+1,
  \end{displaymath}
  for $j = 0, \dots, 2K+1$, and the corresponding eigenvalues by
  \begin{displaymath}
    \tilde\nu_j = 4 \sin^2 \big( j \pi / (4K+4) \big),
    \quad j = 0, \dots, 2K+1.
  \end{displaymath}
  The first eigenvector $g^{(0)}$ is constant, and hence all other
  eigenvectors have mean zero. This implies that the eigenvalues
  $\nu_j$, $j = 1, \dots, 2K+1,$ give the remaining eigenvalues of
  $\Lqnlconj:\R^{2N}_*\to \R^{2N}_*$. This concludes the proof of the
  lemma.
\end{proof}

\begin{remark}
  Even though Lemma \ref{th:spec_Lqnl_U12} gives uniform bounds on the
  spectrum of $\Lqnl$ in $\Us^{1,2},$ it does not give the desired
  sharper result that eigenvalues are clustered, for example, at
  $A_F$. As a matter of fact, Lemma \ref{th:spec_Lqnl_U12} shows that
  this is never the case. However, we see that, if $K$ remains bounded
  as $N \to \infty$, then all but a finite number of eigenvalues of
  $\Li^{-1/2}\Lqcf \Li^{-1/2}$ are identically equal to $A_F$.
\end{remark}

\begin{figure}[t]
  \begin{center}
    \includegraphics[width=11cm]{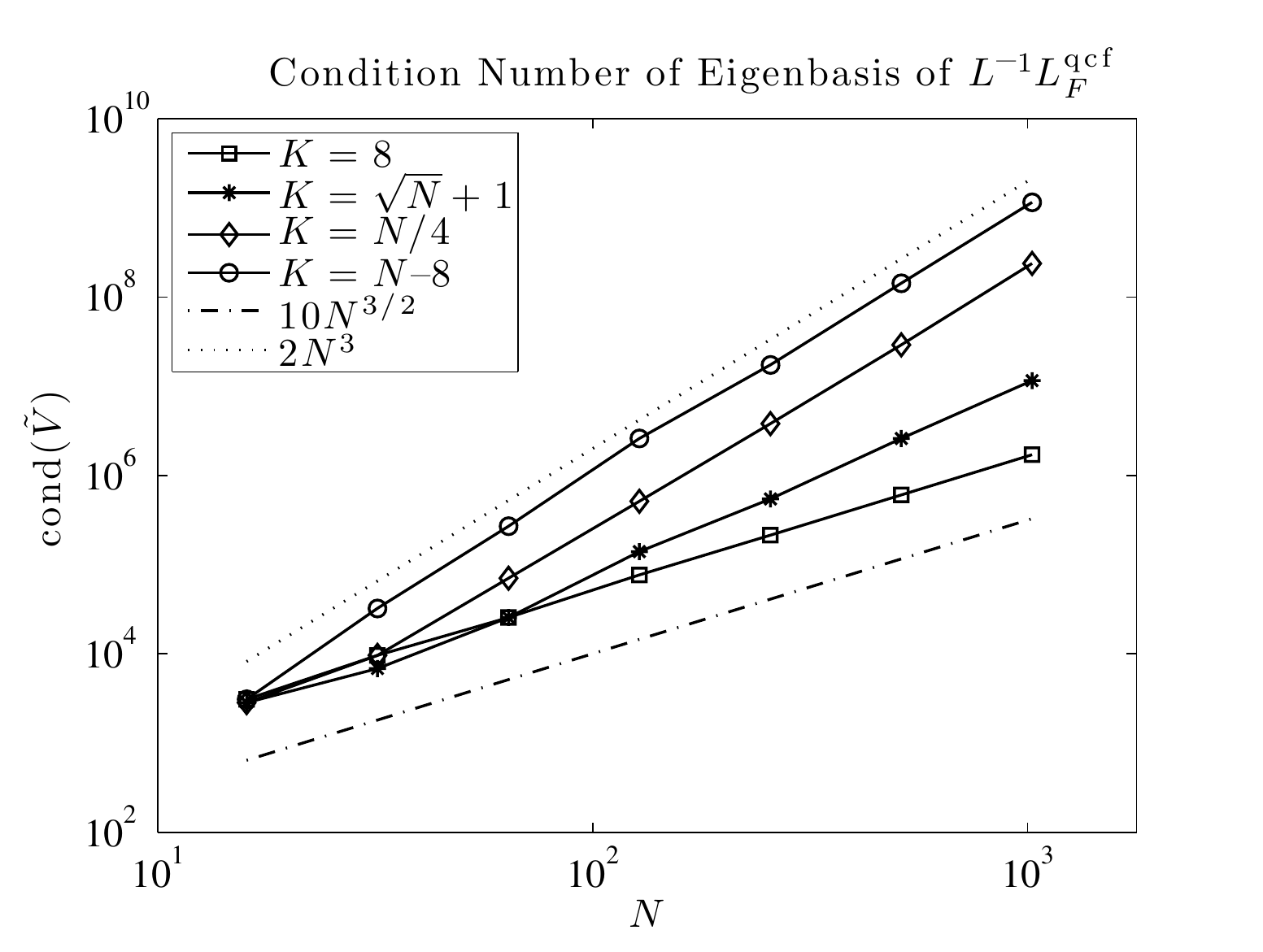}
  \end{center}
  \caption{\label{fig:qcfCondP} Condition number of the matrix $\Vqcf$
    plotted against the system size $N$ for $A_F/\phi''_F =
    0.4,$ and various atomistic region sizes $K,$ where $\Li^{-1}
    \Lqcf = \Vqcf \widetilde \Lambda^{{\rm qcf}} \Vqcf^{-1}$ is the
    spectral decomposition of $\Li^{-1}\Lqcf$.
    Since $(\phi''_F)^{-1}\Lqcf$ depends only
on $A_F/\phi''_F$ and $N,$ the matrix $\Vqcf$ depends only on $A_F/\phi''_F$
and $N.$ For each curve we have
    $\cond(\Vqcf$) is O($N^3$), but in fact the curves appear to grow
    like $N^{3/2} K^{3/2}.$
    }
\end{figure}

\begin{figure}[t]
  \begin{center}
    \includegraphics[width=11cm]{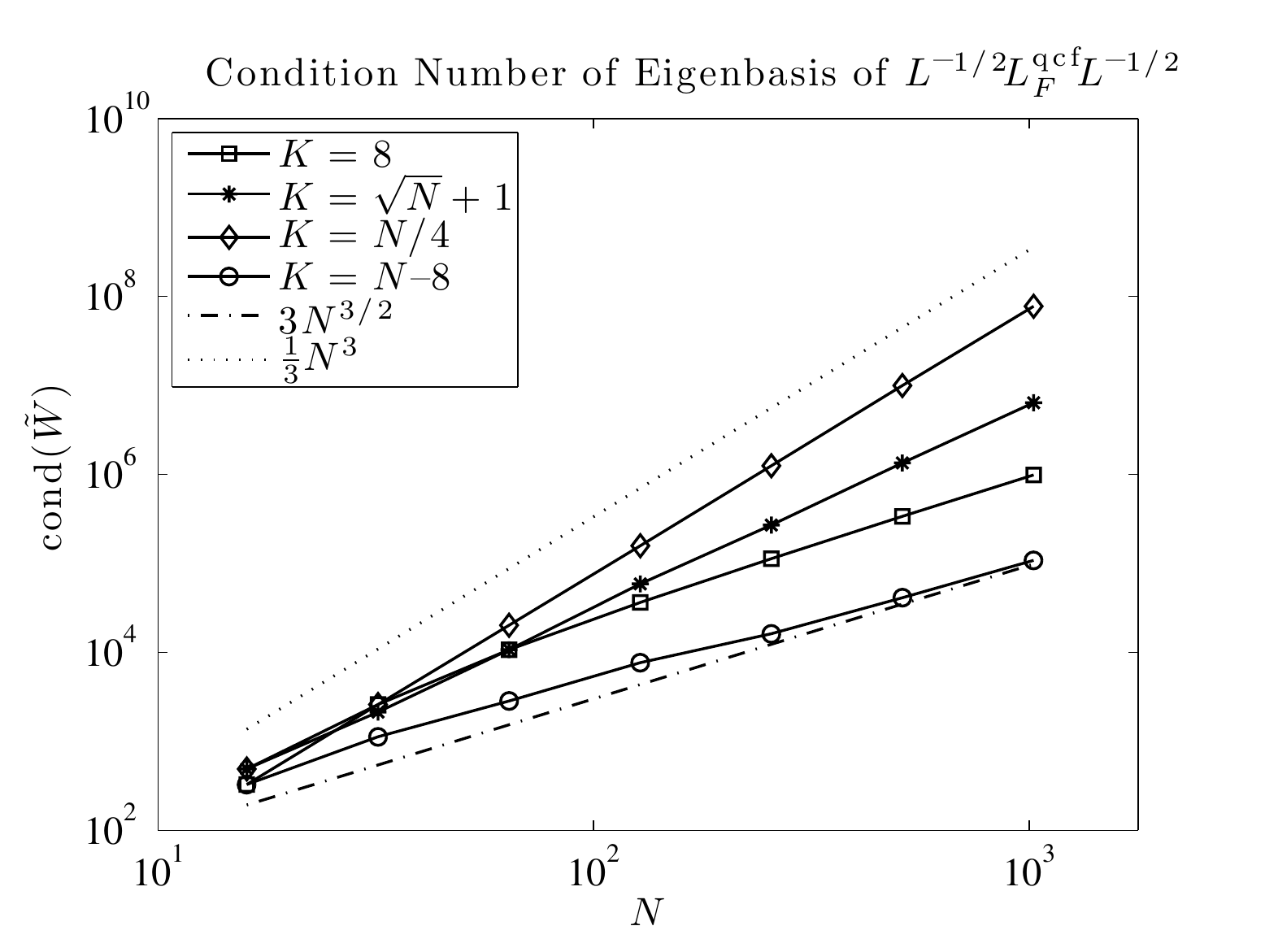}
  \end{center}
  \caption{\label{fig:qcfCondPw} Condition number of the matrix
    $\Wqcf$ plotted against the system size $N$ for $A_F/\phi''_F =0.4,$
    and various atomistic region sizes $K,$ where
    $\Li^{-1/2} \Lqcf \Li^{-1/2} = \Wqcf \widetilde \Lambda^{{\rm
        qcf}} \Wqcf^{-1}$ is the spectral decomposition of $\Li^{-1/2}
    \Lqcf \Li^{-1/2}$. For each curve, $\cond(\Wqcf$) is O($N^3$).}
\end{figure}

\medskip

We conclude this study by considering the condition number of the
matrix of eigenvectors for the eigenvalue problems~\eqref{eigv}
and~\eqref{eigw}.  We write $\Li^{-1}\Lqcf = {\Vqcf} \tilde
\Lambda^{{\rm qcf}} {\Vqcf}^{-1},$ where $\tilde \Lambda^{{\rm qcf}}$
is the diagonal matrix of eigenvalues of $\Li^{-1}\Lqnl$ and $\Vqcf$
is the associated matrix of eigenvectors.  In
Figure~\ref{fig:qcfCondP}, we have plotted numerical results for the
condition number of the matrix $\Vqcf.$
We note that great care must
be taken when computing the basis of eigenvectors since one eigenvalue
has a high multiplicity (cf. Lemma \ref{th:spec_Lqnl_U12}).  As
described in Appendix~\ref{app:cond_calc}, the block structure of the
matrix $\Li^{-1}\Lqcf$ allows us to analytically compute most of the
eigenvectors corresponding to the high multiplicity eigenvalue and to
separately compute all remaining eigenvectors.

The numerical experiment displayed in Figure~\ref{fig:qcfCondP} leads
to the following conjecture.

\begin{conjecture}
  \label{th:qcf:condP_U02}
  Let $\Vqcf$ denote the matrix of eigenvectors for the preconditioned
  force-based QC operator $\Li^{-1} \Lqcf$.  If $A_F > 0$, then
  ${\cond}(\Vqcf) = O\left(N^{3} \right)$ as $N \to \infty$.
\end{conjecture}

\medskip It follows from \eqref{eigv} and \eqref{eigw} that we can
write $\Li^{-1/2} \Lqcf \Li^{-1/2} = {\Wqcf} \tilde \Lambda^{{\rm
    qcf}} {\Wqcf}^{-1}$ where $\Wqcf=\Li^{1/2}\Vqcf$ is the associated
matrix of eigenvectors.  In Figure~\ref{fig:qcfCondPw}, we have
plotted numerical results for the condition number of the matrix
$\Wqcf.$ These calculations can be simplified by observing that, if we
define the operator $D: \R^{2N-1} \rightarrow \R^{2N}$ by $D v := v'$
then $\Wqcf^T \Wqcf = \Vqcf^T L \Vqcf = \Vqcf^T D^T D \Vqcf.$ Since
the condition number of a matrix $A$ depends only on the eigenvalues
of $A^T A$, it follows that $\cond(D\Vqcf)=\cond(\Wqcf)$.

The numerical experiment displayed in Figure~\ref{fig:qcfCondPw} leads
to the following conjecture.

\begin{conjecture}
  \label{th:qcf:condP_U02w}
  Let $\Wqcf$ denote the matrix of eigenvectors for the preconditioned
  force-based QC operator $\Li^{-1/2} \Lqcf \Li^{-1/2}.$ If
  $A_F>0$, then ${\cond}(\Wqcf) = O\left(N^{3} \right)$ as
  $N \to \infty$.
  \end{conjecture}

\section{Iterative Methods for the Nonlinear QCF System}
\label{sec:nonlinear}
In this section, we briefly review and analyze two common solution
methods for the QCF equilibrium equations. The first method, the {\em
  ghost force correction (GFC) scheme}, is often considered an
independent approximation scheme rather than an iterative method for
the solution of the QCF system. However, it was shown in
\cite{Dobson:2008a} that the ghost force correction, when iterated to
self-consistency, does in fact give rise to the QCF method. In the
following section, we will show that a linearization of the GFC method
predicts a lattice instability at a strain significantly less than the
critical strain of the atomistic model.

The second method that we discuss solves the QCF equilibrium equations
by computing the location along the search direction where the
residual is orthogonal to the search direction~\cite{Miller:2008}.
We show in Section~\ref{sec:modifiedcg} that the indefiniteness of
$\Lqcf$ implies that this method cannot be expected to be numerically
stable for the QCF system.

\subsection{The Ghost Force Correction}
\label{sec:gfc_method}
After discovering that the original energy-based QC method (QCE) is
inconsistent at the interface, a dead load correction was proposed to
remove the so-called {\em ghost forces}~\cite{Shenoy:1999a}.  The idea
of this {\em ghost force correction (GFC)} is the following: Since the
Cauchy--Born continuum model is consistent with the atomistic model,
the ``defective'' (inconsistent) forces of the QCE method at the
interface are simply replaced by the Cauchy--Born forces in the continuum
region and by the atomistic forces in the atomistic region.
The discrepancy between the forces of
the QCE method and those of the QCF method are called the {\em ghost
  forces}, and are defined as follows:
\begin{displaymath}
  g(y) := \Fs^{\rm qcf}(y) - \Fs^{\rm qce}(y)
\end{displaymath}
where
\begin{displaymath}
\Fs^{\rm qce}(y) := -\eps^{-1} \nabla \E^{\rm qce}(y).
\end{displaymath}
It is clear that the ghost forces are concentrated in a neighborhood
of the atomistic-to-continuum interface and can therefore be computed
efficiently \cite{Shenoy:1999a}. The GFC is then normally applied
during a quasistatic loading process. In the following example
algorithm, the loading parameter is the macroscopic strain $F>0$ and
the corresponding space of admissible deformations is
$
\Ys_F=y^F+\Us.
$

\medskip \noindent {\bf GFC Iteration: } \\[-4mm]
\begin{enumerate}
\item[{\bf 0.}] Input: $y^{(0)} \in \Ys_1$ such that $\Fs^{\rm
    qcf}(y^{(0)}) + f \approx 0$; increment $\delta F > 0$
\item[{\bf 1.}] For $n = 1, 2, 3, \dots$ do
\item[{\bf 2.}] \qquad Evaluate $g^{(n)} = g(\hat{y}^{(n-1)})$, where
  $\hat{y}^{(n-1)} = y^{(n-1)} + x \delta F $
\item[{\bf 3.}] \qquad Find $y^{(n)} \in \argmin \,
  \big\{\Eqce(y) - \< f, y \> - \< g^{(n)}, y \> : y \in \Ys_{1 + n\delta F} \big\}.$
\end{enumerate}

\begin{remark}
  Increased efficiency can be obtained by allowing nonuniform steps
  and multiple GFC iterations at a fixed load~\cite{dobsonluskin08},
  thus introducing a second inner loop. For the purpose of the present
  paper, we will focus on the simpler algorithm above.
\end{remark}

\medskip
We now consider the GFC iteration above for purely
tensile loading which is given by $f=0.$  We also take the initial
iterate to be the uniform deformation for $F=1,$ that is, $y^{(0)}=y^1=x,$
and $\delta F$ to be small.
Then it is easy to see that the GFC iteration gives the uniform deformation
$y^{(n)}=y^{1+n\delta F}$ until $1+n\delta F>F^{\rm gfc},$ where
$F^{\rm gfc}$ is the uniform strain at which $\Lqce$ becomes
unstable.  We recall from Lemma~\ref{th:stab_qce} that $\Lqce$ becomes
unstable at $F^{\rm gfc}$ satisfying
\[
A_{F^{\rm gfc}} + \lambda_K \phi_{2F^{\rm gfc}}''=0,
\]
where $\smfrac12 \leq \lambda_K \leq 1$ and
$\phi_{2F^{\rm gfc}}''<0,$ so $F^{\rm gfc}<F_*.$

The critical strain $F^{\rm qce}$ for the uncorrected  energy
$\Eqce(y)$ was
investigated in \cite{doblusort:qce.stab} by linearizing $\Eqce(y)$
about
\[
y_{\rm qce}^F\in \argmin \,
  \big\{\Eqce(y)  : y \in \Ys_{F} \big\}
  \]
rather than about $y^F.$  It was shown, in agreement with the
computational experiments in \cite{doblusort:qce.stab}
and \cite{Miller:2008}, that the GFC method
does improve the accuracy of the computation for the critical strain,
that is,
\[
F^{\rm qce}<F^{\rm gfc}<F_*.
\]
  See \cite{doblusort:qce.stab}
for a more precise statement of these results.

\subsection{A modified conjugate gradient method}
\label{sec:modifiedcg}
Another popular approach to solving the QCF equilibrium equations is
to replace the univariate optimization used for step size selection in
the nonlinear conjugate gradient method~\cite{NocedalWright99} with
the computation of a step size where the residual is orthogonal to the
current search direction~\cite{Miller:2008}.
More specifically, if $d^{(n)}$ is the current search direction, then
this method computes $y^{(n+1)} = y^{(n)}+\alpha^{(n)} d^{(n)}$ such
that
\begin{equation} \label{miller}
  \big\<\Fs^{\rm qcf}(y^{(n+1)}) + f,\, d^{(n)}\big\> \approx 0.
\end{equation}

We can easily see that this method is numerically unstable by
considering a linearization of~\eqref{miller} about the uniform
configuration $y^F$ to obtain
\begin{equation*}
  \big\<-\Lqcf\big(u^{(n)}+\alpha^{(n)} d^{(n)}\big)
  + f ,\, d^{(n)}\big\>= 0,
\end{equation*}
or equivalently,
\begin{equation*}
  - \alpha^{(n)} \big\< \Lqcf d^{(n)},\, d^{(n)} \big\>
  + \big\< \Lqcf u^{(n)},\, d^{(n)} \big\>
  + \big\< f,\, d^{(n)} \big\> = 0.
\end{equation*}

However, according to Theorem~\ref{th:nocoerc}, $\Lqcf$ is indefinite,
which implies that there exist directions $d$ such that $\< \Lqcf d, d
\> = 0$. Hence, if such a singular direction $d$ is chosen (for
example, if the initial iterate satisfies $\Lqcf u^{(0)} = d$) then the
step size $\alpha^{(n)}$ is undefined. More generally, if a direction
$d^{(n)}$ is ``near'' such a singular direction (for example, $\Lqcf
u^{(0)} \approx d$), then the computation of $\alpha^{(n)}$ is
numerically unstable.

\section{GMRES Solution of the Linear QCF Equations}
\label{sec:gmres}
We now consider the generalized minimal residual method (GMRES) to
find (approximate) solutions to the linear, force-based QC equilibrium
equations
\begin{equation}
  \Lqcf \uqcf = f
  \label{eq:Lqcf_eq}.
\end{equation}
GMRES is an attractive iterative method for the solution of
nonsymmetric linear equations since the iterates satisfy a minimality
property for the residual.  This minimality property is the basis for
our analysis of the convergence of the GMRES method for the solution
of the QCF equations.

\subsection{Standard GMRES}
\label{sec:gmres_standard}
We recall that GMRES~\cite{saad} builds a sequence of Krylov subspaces
\begin{equation*}
  \Ks_m := \operatorname{span}\Big\{ r^{(0)} ,\, \Lqcf r^{(0)},
    ({\Lqcf})^2 r^{(0)}, \dots, \,({\Lqcf})^{m-1} r^{(0)}\Big\},
\end{equation*}
where $r^{(0)}:= f - \Lqcf u^{(0)}$ is the initial residual, and it
finds an approximate solution
\begin{equation}
  \label{eq:gmres}
  u^{(m)} := \argmin_{v \in u^{(0)}+\Ks_m} \big\|
  f - \Lqcf v \big\|_{\ell^2_\eps}
\end{equation}
that minimizes the $\ell^2_\eps$-norm of the residual $r^{(m)} := f -
\Lqcf u^{(m)}$ for~\eqref{eq:Lqcf_eq}. The residual $r^{(m)}$
satisfies the minimality property
\begin{equation}
  \label{res}
  \big\|r^{(m)}\big\|_{\ell^2_\eps}=\min_{v\in u^{(0)}+\Ks_m}
  \big\| f - \Lqcf v \big\|_{\ell^2_\eps}
  = \min_{\substack{p_{m} \in \P_{m}\\ p_m (0) = 1}}
  \big\|p_m(\Lqcf)r^{(0)}\big\|_{\ell^2_\eps}
\end{equation}
where
\begin{equation*}
\P_m = \{\text{polynomials $p$ of degree }\leq m \}.
\end{equation*}
It follows from \eqref{res} that $r^{(m)}$ depends only on $r^{(0)},$
$A_F/\phi_F'', N,$ and $K.$

GMRES solves the minimization problem \eqref{eq:gmres} by reducing it
to a least squares problem for the coefficients of an
$\ell^2_\eps-$orthonormal sequence $\{v_1,\,\dots,v_{m+1}\}$ computed
by the Arnoldi process.  For details, see~\cite{saad,tref97}.

The convergence analysis does not require a symmetric matrix, and we
will see that Conjectures~\ref{th:samespec} and~\ref{th:qcf:cond_U02}
regarding the spectrum of eigenvalues and conditioning of eigenvectors
are exactly what is needed for an error analysis of GMRES applied to
$\Lqcf.$

\begin{proposition}
  \label{th:gmres_nop}
  If Conjecture~\ref{th:samespec} holds, then
\begin{equation}\label{slow}
\begin{split}
\|r^{(m)}\|_{\ell^2_\eps}&
\le 2 \cond(V) \left(\frac{1-\frac1N\sqrt{\frac{2A_F}{\phi_{F}''}}}{1+\frac 1N\sqrt{\frac{2A_F}{\phi_{F}''}}}\right)^{m}
\|r^{(0)}\|_{\ell^2_\eps}.
\end{split}
\end{equation}
\end{proposition}

\remark{We recall from Conjecture~\ref{th:qcf:cond_U02} that
$\cond(V)=o\left(\log(N)\right).$ We note that the estimate~\eqref{slow} gives a reduction of the
convergence rate for strains near the critical strain $A_{F_*}=0.$}
\begin{proof}
By Conjecture~\ref{th:samespec}, $\Lqcf$ is diagonalizable, and we
have that $\Lqcf = V \Lambda^{{\rm qcf}} V^{-1}$ where $V$ contains the
eigenvectors of $\Lqcf$ as its columns and where $\Lambda^{{\rm qcf}}$ is the
diagonal matrix of eigenvalues of $\Lqcf.$ We denote the set of
eigenvalues of $\Lqcf$ by $\sigma(\Lqcf).$ We then have by
~\eqref{res} that
\begin{equation*}
  \begin{split}
    \|r^{(m)}\|_{\ell^2_\eps}
    &= \min_{\substack{p_{m} \in \P_{m}\\p_m (0) = 1}}
    \big\|p_m(\Lqcf)r^{(0)}\big\|_{\ell^2_\eps}
    = \min_{\substack{p_{m} \in \P_{m}\\ p_m (0) = 1}}
    \big\|Vp_m(\Lambda^{{\rm qcf}})V^{-1}r^{(0)}\big\|_{\ell^2_\eps}\\
    &\le \cond(V) \inf_{\substack{p_{m} \in \P_{m}\\ p_m (0) = 1}}
    \big\| p_m \big\|_{\sigma(\Lqcf)} \big\|r^{(0)}\big\|_{\ell^2_\eps}
\end{split}
\end{equation*}
where
\begin{equation*}
\| p_m \|_{\sigma(\Lqcf)} =
    \sup_{\lambda \in \sigma(\Lqcf)} | p_m(\lambda) |.
\end{equation*}

By Conjecture~\ref{th:samespec}, $\Lqcf$ and $\Lqnl$ share the same
spectrum, so we have that
\[
\inf_{\substack{p_{m} \in \P_{m}\\
p_m (0) = 1}} \| p_m \|_{\sigma(\Lqcf)}
=\inf_{\substack{p_{m} \in \P_{m}\\
p_m (0) = 1}} \| p_m \|_{\sigma(\Lqnl)}
\le
\inf_{\substack{p_{m} \in \P_{m}\\
p_m (0) = 1}}\quad\max_{\lqnl_{1} \leq \lambda \leq \lqnl_{2N-1}} |p_m(\lambda)|.
\]
We now recall ~\cite{saad} that
$$
\inf_{\substack{p_m \in \P_{m}\\
p_m (0) = 1}}\quad\max_{\lqnl_{1} \leq \lambda \leq \lqnl_{2N-1}} |p_m
(\lambda)|
\leq
2\left(\frac{1-\sqrt{\gamma}}{1+\sqrt{\gamma}}\right)^{m}
$$
where $\gamma=1/\cond(\Lqnl)=\lqnl_{1}/\lqnl_{2N-1}.$
We have
by Lemma~\ref{qnlcond} that $\gamma
\le(2A_F\eps^2)/\phi_{F}''.$
It thus follows that
\begin{equation*}
\begin{split}
\|r^{(m)}\|_{\ell^2_\eps}&\le
2 \cond(V) \left(\frac{1-\sqrt{\gamma}}{1+\sqrt{\gamma}}\right)^{m}
\left\|r^{(0)}\right\|_{\ell^2_\eps}
\\
&\le 2 \cond(V) \left(\frac{1-\eps\sqrt{\frac{2A_F}{\phi_{F}''}}}{1+\eps\sqrt{\frac{2A_F}{\phi_{F}''}}}\right)^{m}
\|r^{(0)}\|_{\ell^2_\eps}.\qedhere
\end{split}
\end{equation*}
\end{proof}

In Figures \ref{fig:gmres_res} and \ref{fig:gmres_err}, we display the
residual and error of the standard GMRES iterates when the algorithm
is applied to the solution of the QCF system with right-hand side
\begin{equation}
  \label{eq:num_ex_rhs}
  f(x) = h(x) \cos(3 \pi x) \quad \text{where} \quad
  h(x) = \begin{cases}
    \ \ 1, & x \geq 0, \\
    -1, & x < 0,
  \end{cases}
\end{equation}
which is smooth in the continuum region but has a discontinuity in the
atomistic region.
We also set $A_F = 0.5$ and $\phi_F'' = 1$. We observe the
slow convergence predicted by the theory of this section. However, we
also observe alternation of slow and fast regimes, which our theory
was unable to predict.

\begin{figure}
  \begin{center}
    \includegraphics[width=11cm]{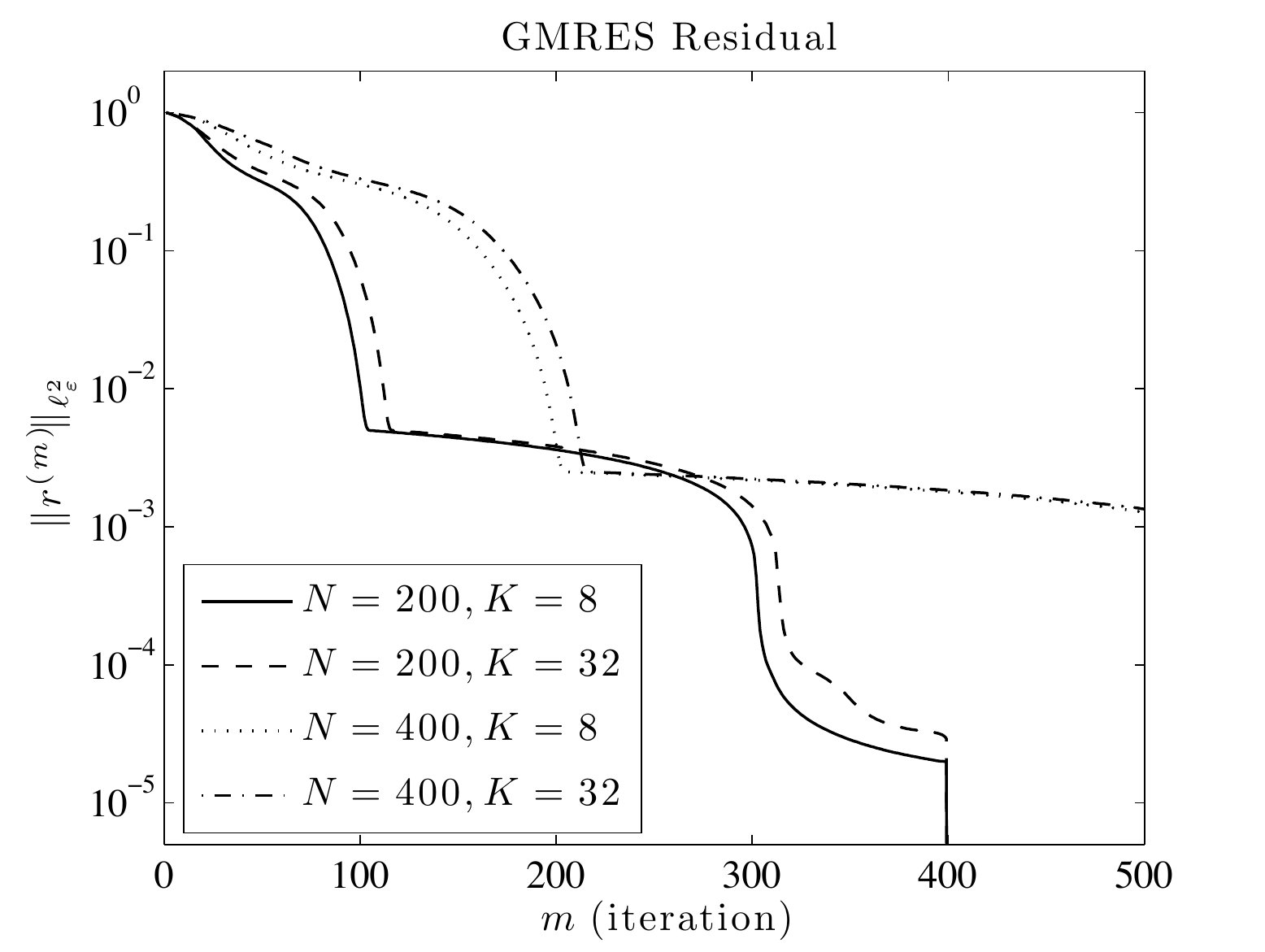}
  \end{center}
  \caption{\label{fig:gmres_res} Application of standard GMRES to the
    QCF system \eqref{eq:Lqcf_eq} with right-hand side
    \eqref{eq:num_ex_rhs}, $A_F = 0.5,$ and $\phi_F'' = 1$. We plot
    the $\ell^2_\eps$-norm of the residual against the iteration
    number $m$ for various choices of $N$ and $K$. We observe the slow
    convergence of the residual partially predicted by the theory in
    section \ref{sec:gmres_standard}.  We recall that there are $2N-1$
    degrees of freedom. }
\end{figure}

\begin{figure}
  \begin{center}
    \includegraphics[width=11cm]{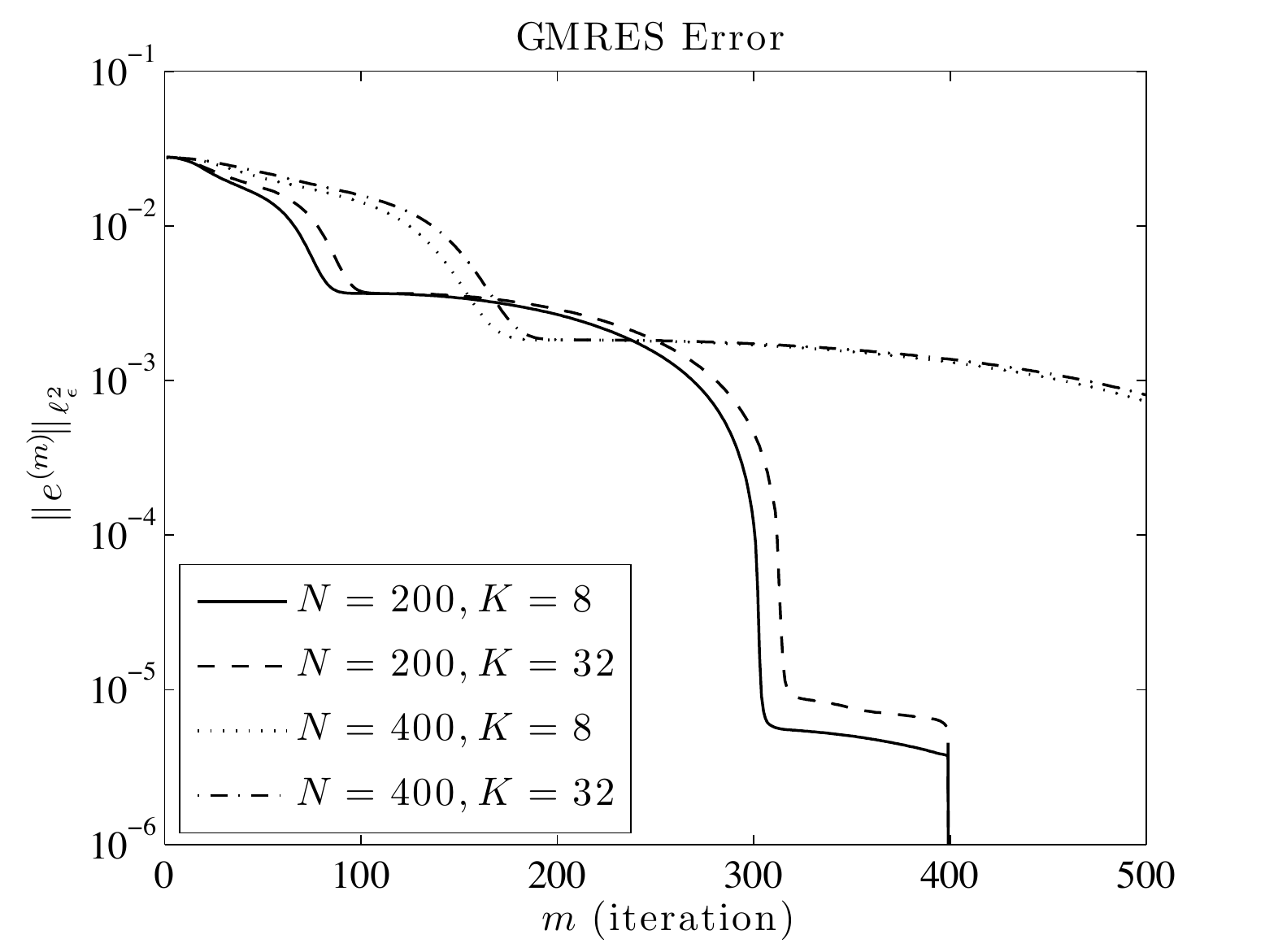}
  \end{center}
  \caption{\label{fig:gmres_err} Application of standard GMRES to the
    QCF system \eqref{eq:Lqcf_eq} with right-hand side
    \eqref{eq:num_ex_rhs}, $A_F = 0.5,$ and $\phi_F'' = 1$. We plot
    the $\ell^2_\eps$-norm of the error $e^{(m)} = u^{(m)} - u^{\rm
      qcf}$ against the iteration number $m$ for various choices of
    $N$ and $K$. We observe that $\|e^{(m)}\|_{\ell^2_\eps}$ closely
    mirrors the norm of the residual $\|r^{(m)}\|_{\ell^2_\ell}$. }
\end{figure}

\subsection{Preconditioned GMRES with $P=\Li$}
\label{sec:P-GMRES}
We next consider the GMRES algorithm left-preconditioned by $P=\Li,$
which is the GMRES algorithm applied to the left-preconditioned QCF
equilibrium equations~\cite{saad}
\begin{equation}
\Li^{-1}\Lqcf \uqcf = \Li^{-1}f
\label{eq:Lqcf_eqp}.
\end{equation}
We now denote the $m$th left-preconditioned Krylov subspace by
\begin{equation*}
  \tilde \Ks_m =: \operatorname{span} \Big\{
  \Li^{-1}r^{(0)},\, \big(\Li^{-1}\Lqcf\big)
  \Li^{-1}r^{(0)},
  \dots, \,
  \big(\Li^{-1}{\Lqcf}\big)^{m-1} \Li^{-1}r^{(0)}\Big\}
\end{equation*}
and compute the minimizer
\begin{equation*}
  u^{(m)} := \argmin_{v \in u^{(0)}+\tilde\Ks_m} \big\|
  \Li^{-1}\big(f - \Lqcf v\big)\big\|_{\ell^2_\eps}.
\end{equation*}

\begin{proposition}
  \label{th:convergence_gmres_p}
  If Conjecture~\ref{th:spec_U12} holds, then
  \begin{equation}\label{vest}
    \begin{split}
      \big\|\Li^{-1}r^{(m)}\big\|_{\ell^2_\eps}&\le
      2 \cond(\Vqcf) \left(\smfrac{1 - \sqrt{\frac{A_F}{\phi''_F}}}{1 + \sqrt{\frac{A_F}{\phi''_F}}}\right)^m
      \big\|\Li^{-1}r^{(0)}\big\|_{\ell^2_\eps}.
    \end{split}
  \end{equation}
\end{proposition}
\remark{We recall that Conjecture~\ref{th:qcf:condP_U02} states that
  $\cond(\Vqcf)=O\left(N^3\right).$}
\begin{proof}
  As in the proof of Proposition \ref{th:gmres_nop} above, the
  residual satisfies
  \begin{equation}\label{resp}
    \begin{split}
      \left\|\Li^{-1}r^{(m)}\right\|_{\ell^2_\eps}&=\min_{v\in u^{(0)}+\tilde\Ks_m}\left \| \Li^{-1}(f - \Lqcf v) \right\|_{\ell^2_\eps}\\
      &=\min_{\substack{p_{m} \in \P_{m}\\
          p_m (0) = 1}}  \left\|p_m\left(\Li^{-1}\Lqcf\right)\Li^{-1}r^{(0)}\right\|_{\ell^2_\eps}\\
      &=\min_{\substack{p_{m} \in \P_{m}\\
          p_m (0) = 1}} \left\|\Vqcf p_m(\tilde\Lambda^{{\rm qcf}}){\Vqcf}^{-1}\Li^{-1}r^{(0)}\right\|_{\ell^2_\eps}\\
      &\le \cond(\Vqcf)
      \inf_{\substack{p_{m} \in \P_{m}\\
          p_m (0) = 1}} \left\| p_m \right\|_{\sigma\left(\Li^{-1}\Lqcf\right)}\left\|\Li^{-1}r^{(0)}\right\|_{\ell^2_\eps}\\
    \end{split}
  \end{equation}
  where $\Vqcf$ is a matrix with the eigenvectors of $\Li^{-1}\Lqcf$
  as its columns and ${\Vqcf}^{-1} \Li^{-1}\Lqcf \Vqcf$ is the
  diagonal matrix $\tilde\Lambda^{{\rm qcf}}.$ By
  Conjecture~\ref{th:spec_U12}, $\Li^{-1} \Lqcf$ has the same spectrum
  as $\Li^{-1} \Lqnl,$ and by Lemma~\ref{th:spec_Lqnl_U12}, we have
  that $\tilde \gamma := \mu_1^{\rm qnl} / \mu_{2N-1}^{\rm qnl} \ge
  {A_F}/{\phi''_F}.$ Using the bound on the spectrum, we arrive at the
  estimate ~\eqref{vest}.
  It follows from \eqref{resp} that $\Li^{-1}r^{(m)}$ depends only on $\Li^{-1}r^{(0)},$
$A_F/\phi_F'',$ and $N.$
\end{proof}

Numerical experiments describing the convergence of the preconditioned
GMRES method are displayed in Figures \ref{fig:gmres_p_res} and
\ref{fig:gmres_p_err}. In the first iteration, we observe a large
decrease in the residual, which can be explained by the fact that $1$
is a multiple eigenvalue. Next, we see that the iteration for the two
cases with $K = 4$ converges to machine precision in 10
iterations. This is an immediate consequence of Lemma
\ref{th:spec_Lqnl_U12}, which shows that $\Li^{-1} \Lqcf$ has exactly
$2K+2$ distinct eigenvalues. Finally, we observe precisely the
convergence rate for the residual predicted in Proposition
\ref{th:convergence_gmres_p}, which is independent of $N$ and
$K$. However, we also notice in Figure \ref{fig:gmres_p_err} that the
error is not directly related to the residual. This may be caused by a
large condition number of the eigenbasis, and means that the residual
is not necessarily a reliable termination criterion. Finally, we note
that, even though in this experiment $A_F$ is close to zero (that is,
the systems is close to an instability), we still observe rapid
convergence of the method.

\begin{figure}
  \begin{center}
    \includegraphics[width=11cm]{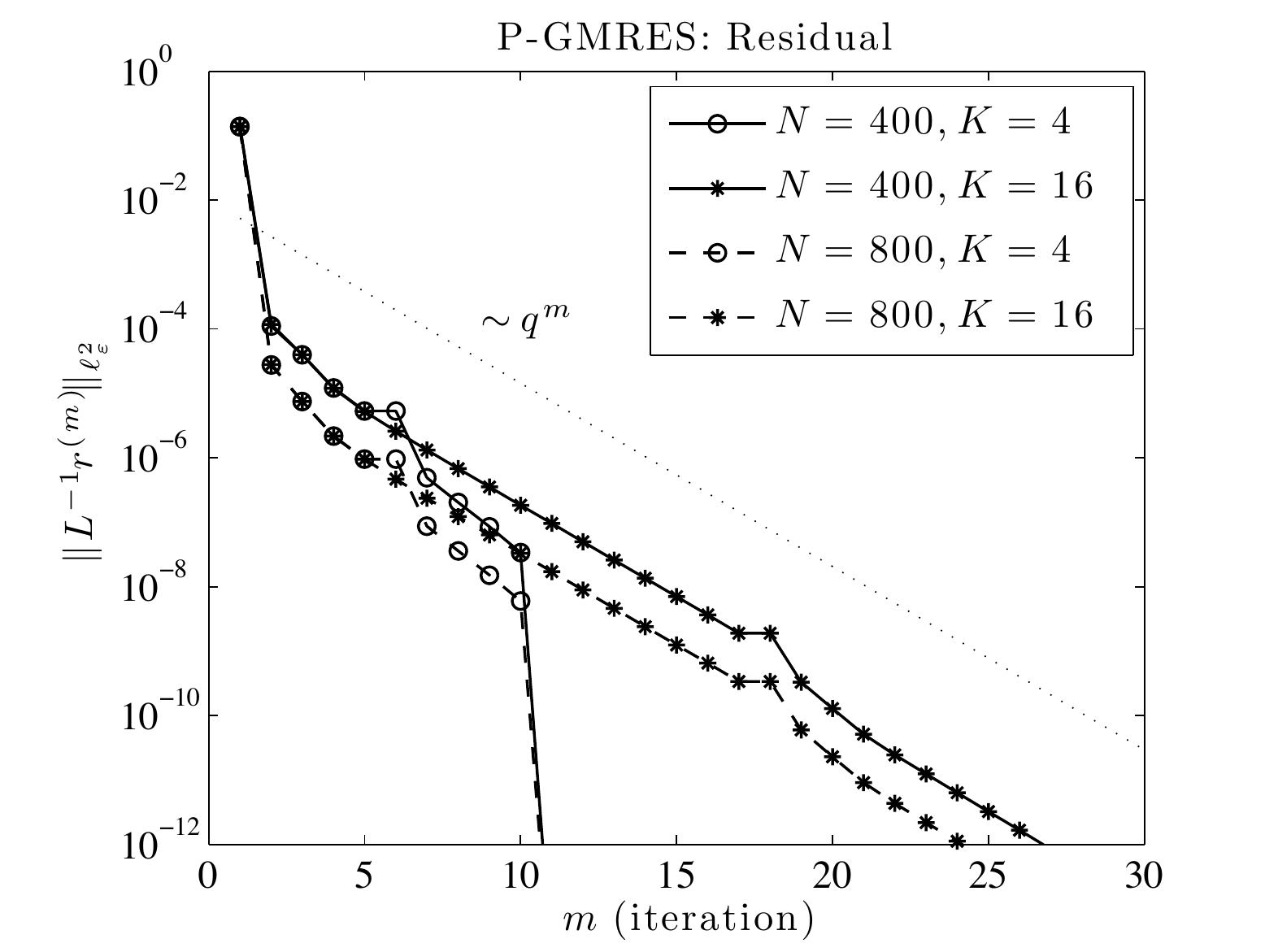}
  \end{center}
  \caption{\label{fig:gmres_p_res} Application of preconditioned GMRES
    to the QCF system \eqref{eq:Lqcf_eq} with right-hand side
    \eqref{eq:num_ex_rhs}, and with $A_F = 0.1$ and $\phi_F'' = 1$. We
    plot the $\ell^2_\eps$-norm of the preconditioned residual against
    the iteration number $m$ for various choices of $N$ and $K$. We
    observe precisely the convergence rate $\|L^{-1}
    r^{(m)}\|_{\ell^2_\eps} \sim q^m$ with $q = (1 -
    \sqrt{A_F/\phi_F''}) / (1+\sqrt{A_F/\phi_F''})$, predicted in
    Proposition \ref{th:convergence_gmres_p}. }
\end{figure}

\begin{figure}
  \begin{center}
    \includegraphics[width=11cm]{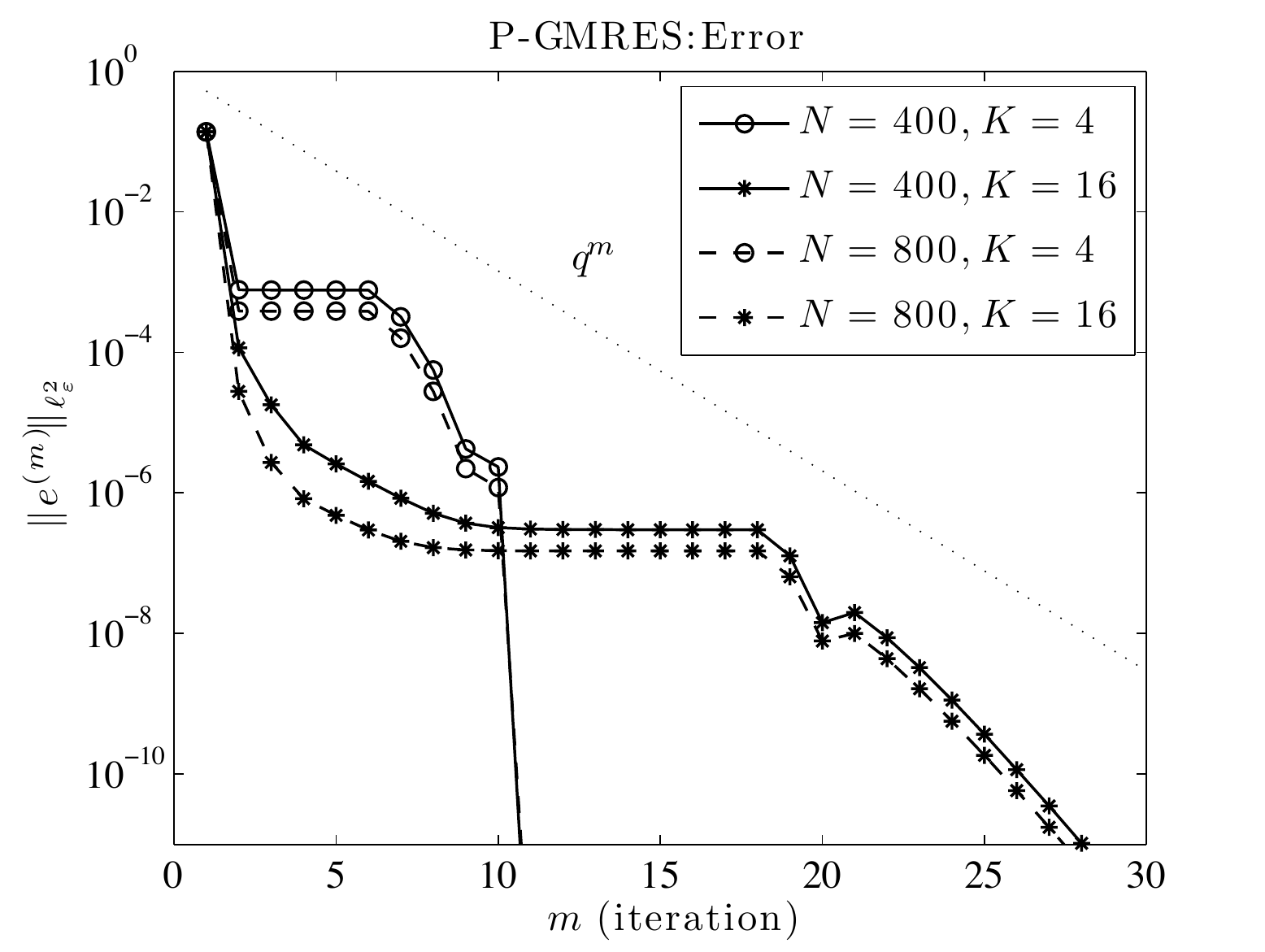}
  \end{center}
  \caption{\label{fig:gmres_p_err} Application of preconditioned GMRES
    to the QCF system \eqref{eq:Lqcf_eq} with right-hand side
    \eqref{eq:num_ex_rhs}, and with $A_F = 0.1$ and $\phi_F'' = 1$. We
    plot the $\ell^2_\eps$-norm of the error $e^{(m)} = u^{(m)} -
    u^{\rm qcf}$ against the iteration number $m$ for various choices
    of $N$ and $K$. The expected rate is $\|e^{(m)}\|_{\ell^2_\eps}
    \sim q^m$ where $q = (1 - \sqrt{A_F/\phi_F''}) /
    (1+\sqrt{A_F/\phi_F''})$. }
\end{figure}

\subsection{Preconditioned GMRES with $P=\Li$ in the $\Us^{1,2}$ norm}
\label{sec:PQ-GMRES}
A possible reason for the poor connection between residual and error
in the preconditioned GMRES method is that we have minimized the
residual in an inappropriate norm.  A more natural norm than $\|
\Li^{-1} r^{(m)}\|_{\ell^2_\eps}$ is the $\Us^{1,2}$-norm ~\eqref{pos} of $\Li^{-1}
r^{(m)}$
\begin{displaymath}
  \| \Li^{-1} r^{(m)} \|_{\Us^{1,2}}= \| \Li^{-1/2} r^{(m)}\|_{\ell^2_\eps}=\|r^{(m)}\|_{\Us^{-1,2}}.
\end{displaymath}
This gives a clear motivation for minimizing the preconditioned
residual $\Li^{-1} r^{(m)}$ in the $\Us^{1,2}$-norm (see also
\cite[Sec. 13]{Simoncini:2007} for a more extensive discussion of this
idea and interesting generalizations).

This leads to a variant of the preconditioned GMRES method where, at
the $m$th step, we compute the minimizer
\begin{equation*}
  u^{(m)} := \argmin_{v \in u^{(0)}+\tilde\Ks_m} \big\| \Li^{-1}\big(f - \Lqcf v\big) \big\|_{\Us^{1,2}},
\end{equation*}
by computing an Arnoldi sequence $\{\tilde v_1,\,\dots,\tilde
v_{m+1}\}$ that is ${\Us^{1,2}}-$orthonormal for the
left-preconditioned equations~\eqref{eq:Lqcf_eqp}.  We then obtain,
subject to the validity of Conjecture~\ref{th:spec_U12}, that the
residuals satisfy
\begin{equation}
  \label{respp}
  \begin{split}
    \big\|\Li^{-1}r^{(m)}\big\|_{\Us^{1,2}}
    &= \min_{v\in u^{(0)}+\tilde\Ks_m} \big\| \Li^{-1}\big(f -
    \Lqcf v\big)\big \|_{\Us^{1,2}} \\
    &= \min_{\substack{p_{m} \in \P_{m}\\p_m (0) = 1}}
    \big\|p_m\big(\Li^{-1}\Lqcf\big)\Li^{-1}r^{(0)}\big\|_{\Us^{1,2}}\\
    &= \min_{\substack{p_{m} \in \P_{m}\\ p_m (0) = 1}}
    \big\|\Vqcf p_m\big(\tilde\Lambda^{{\rm qcf}}\big){\Vqcf}^{-1}
    \Li^{-1}r^{(0)}\big\|_{\Us^{1,2}}\\
    &\le \cond\big(\Li^{1/2}\Vqcf\big)
    \inf_{\substack{p_{m} \in \P_{m}\\ p_m (0) = 1}}
    \big\| p_m \big\|_{\sigma\left(\Li^{-1}\Lqcf\right)}
    \big\|\Li^{-1}r^{(0)}\big\|_{\Us^{1,2}} \\
    &\le 2 \,\cond\big(\Wqcf\big)
    \left(\smfrac{1 - \sqrt{\frac{A_F}{\phi''_F}}}{
        1 + \sqrt{\frac{A_F}{\phi''_F}}}\right)^m
    \big\|\Li^{-1}r^{(0)}\big\|_{\Us^{1,2}}.
  \end{split}
\end{equation}
It follows from \eqref{respp} that $\Li^{-1}r^{(m)}$ depends only on
$\Li^{-1}r^{(0)},$ $A_F/\phi_F'',$ and $N.$

We have thus proven the following convergence result.

\begin{proposition}
  \label{th:convergence_gmres_pq}
  If Conjecture~\ref{th:spec_U12} holds, then
  \begin{equation*}
    \begin{split}
      \big\|\Li^{-1}r^{(m)}\big\|_{\Us^{1,2}}&
      \le 2 \,\cond\big(\Wqcf\big)
      \left(\smfrac{1 - \sqrt{\frac{A_F}{\phi''_F}}}{1 + \sqrt{\frac{A_F}{\phi''_F}}}\right)^m
      \big\|\Li^{-1}r^{(0)}\big\|_{\Us^{1,2}}.
    \end{split}
  \end{equation*}
\end{proposition}
\remark{We recall from Conjecture~~\ref{th:qcf:condP_U02w} that
$\cond(\Wqcf)=O\left(N^3\right).$}

We have tested this variant of the preconditioned GMRES method for the
system \eqref{eq:Lqcf_eq} with right-hand side \eqref{eq:num_ex_rhs}
and displayed the detailed convergence behavior in Figures
\ref{fig:gmres_pq_res} and \ref{fig:gmres_pq_err}. All our
observations about the residual that we made in the previous section
are still valid; in particular, the spectrum of $\Li^{-1}\Lqnl$ (that
is, of $\Li^{-1}\Lqcf$) fully predicts the convergence of the
residual. Moreover, we notice that the residual and the error are now
closely related, that is, the residual can be taken as a reliable
termination criterion for the iterative method. Of course, we have not
presented a proof for this statement and further investigations should
be performed to verify this.

To conclude we remark that, even though we find the GMRES
method in the $\Us^{1,2}$-inner product more attractive from a
theoretical point of view, we have no evidence that it is
considerably more efficient in practice than the more standard
preconditioned GMRES method presented in Section 5.2. As a
matter of fact, additional numerical experiments, the details
of which we do not display here for space reasons,  show that
the decay of the error in the $\Us^{1,2}$-norm is quite similar
for both methods, for a variety of choices of $N$, $K$, and
$f$.

\begin{figure}
  \begin{center}
    \includegraphics[width=11cm]{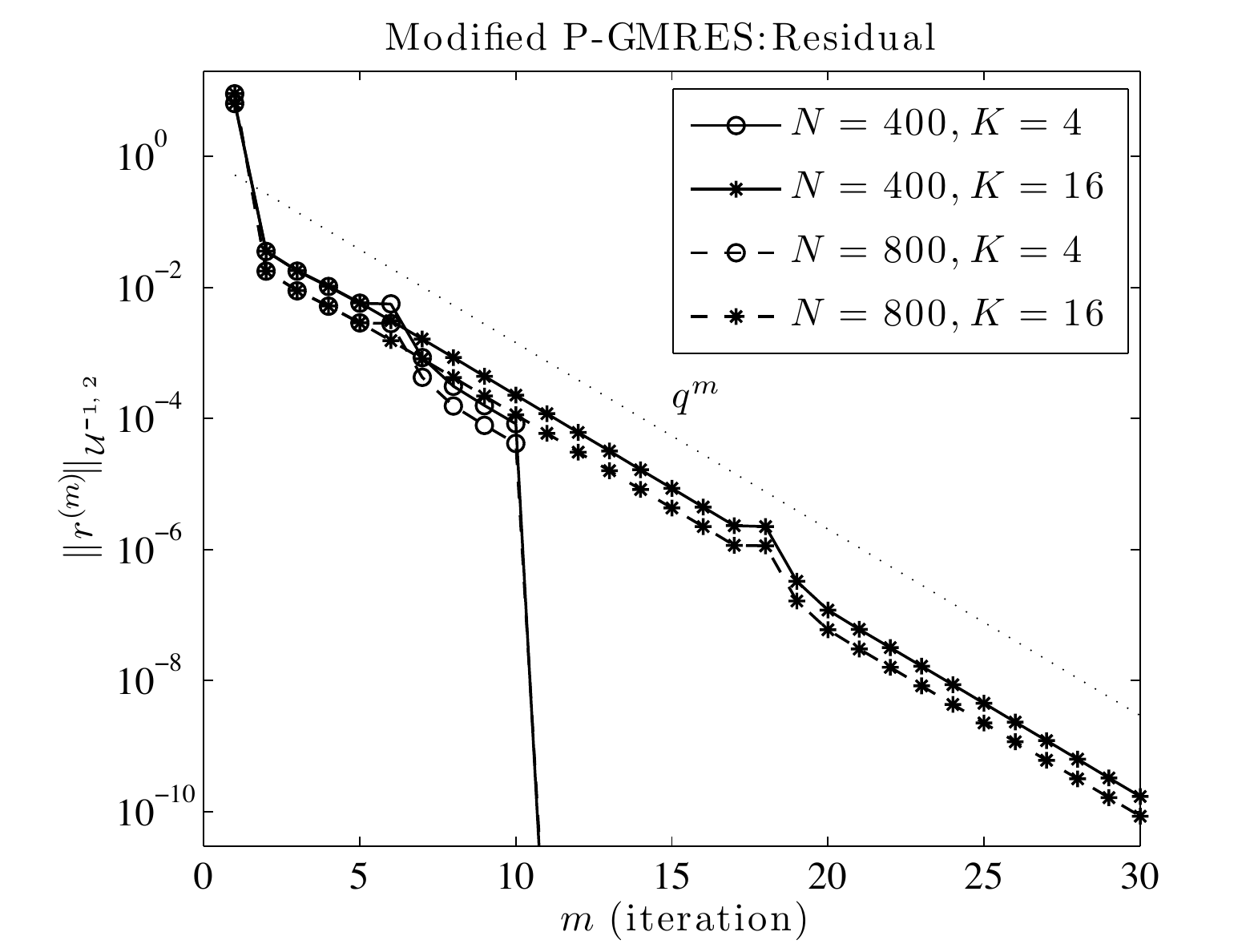}
  \end{center}
  \caption{\label{fig:gmres_pq_res} Application of the preconditioned
    GMRES algorithm with $\Us^{1,2}$-inner product to the QCF system
    \eqref{eq:Lqcf_eq} with right-hand side \eqref{eq:num_ex_rhs}, and with $A_F = 0.1$ and $\phi_F''
    = 1$. We
    plot the $\Us^{-1,2}$-norm ~\eqref{neg} of the residual against the iteration
    number $m$, for various choices of $N$ and $K$. We observe precisely the convergence behaviour predicted by
    Proposition \ref{th:convergence_gmres_pq}, namely $\|
    r^{(m)}\|_{\Us^{-1,2}} \sim q^m$ where $q = (1 -
    \sqrt{A_F/\phi_F''}) / (1+\sqrt{A_F/\phi_F''})$.}
\end{figure}

\begin{figure}
  \begin{center}
    \includegraphics[width=11cm]{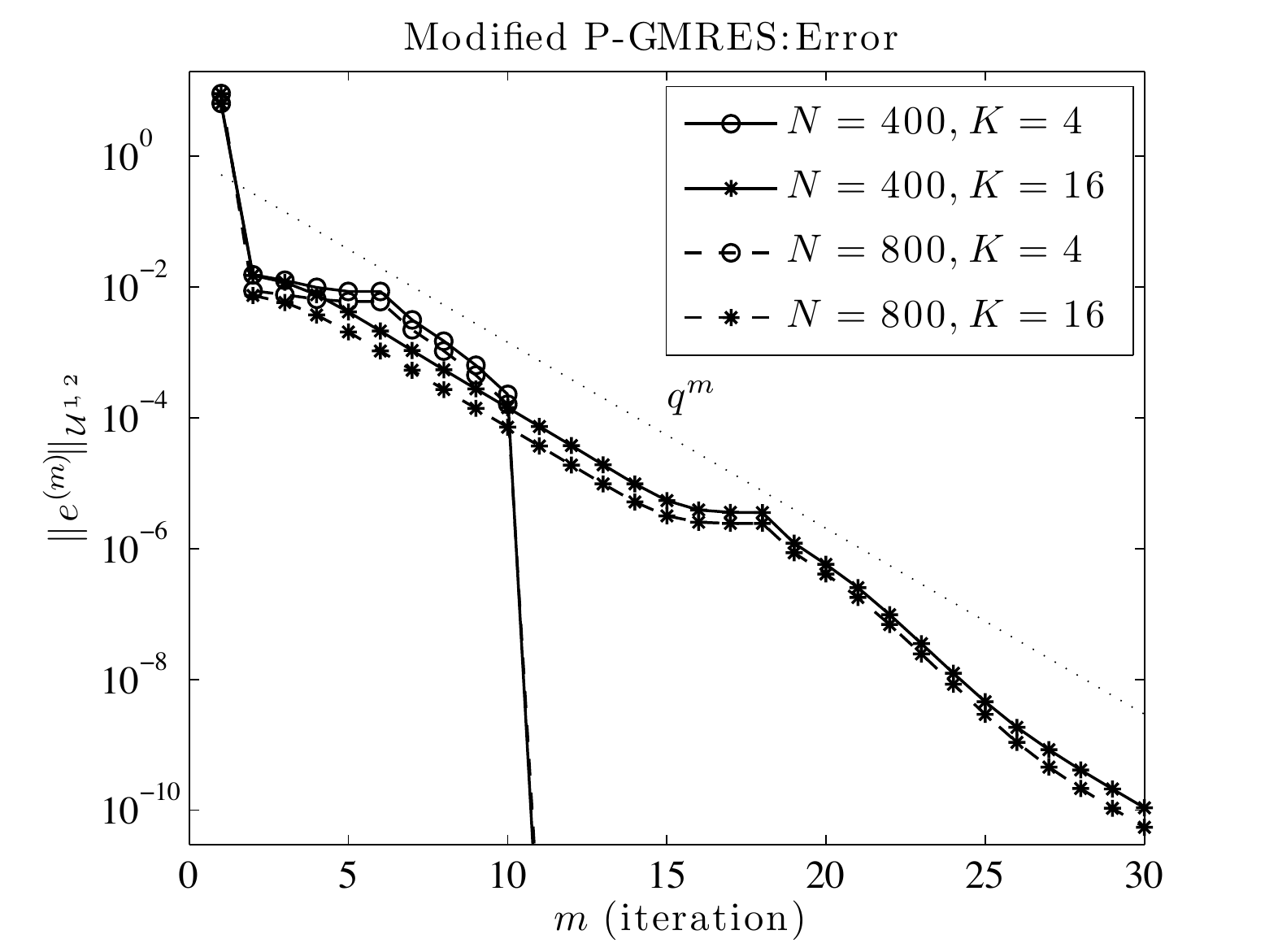}
  \end{center}
  \caption{\label{fig:gmres_pq_err} Application of the preconditioned
    GMRES algorithm with $\Us^{1,2}$-inner product to the QCF system
    \eqref{eq:Lqcf_eq} with right-hand side \eqref{eq:num_ex_rhs}, and
    with $A_F = 0.1$ and $\phi_F'' = 1$. We plot the $\Us^{1,2}$-norm
    of the error $e^{(m)} = u^{(m)} - u^{\rm qcf}$ against the
    iteration number $m$, for various choices of $N$ and $K$. We
    observe that $\| e^{(m)} \|_{\Us^{1,2}}$ closely mirrors the norm
    of the residual $\| \Li^{-1/2} r^{(m)}\|_{\ell^2_\eps},$  that is, the residual
    provides a reliable prediction for the actual error. }
\end{figure}

\section*{Conclusion}

We began by studying the widely used ghost force correction method
(GFC)~\cite{Shenoy:1999a}, which can be understood as a linear stationary method for QCF
using the QCE operator as a preconditioner. We showed that the GFC method
becomes unstable for our model problem before the critical strain
is reached. In practice, this means that the ghost force correction
method would predict a reduced critical strain for the onset of
defect formation or motion.  We also showed that a popular modified
nonlinear conjugate gradient method to solve the QCF equations~\cite{Miller:2008}
is numerically unstable for our model problem.

We then proposed and studied several variants of the generalized minimal
residual method (GMRES), which are a natural choice for the
non-symmetric QCF operator. Since our experience with stationary
methods indicates that the QCL preconditioner combines efficiency and
reliability~\cite{qcf.iterative}, we focused exclusively on this preconditioner. Our
analysis and computational experiments have led us to propose a GMRES
method, which uses the QCL method as a preconditioner as well as the
underlying inner product. This method is reliable for
our model problem up to the critical
strain, and the residual appears to offer a more effective termination criterion.

Future research will
explore the extension of the algorithms and analysis in this paper to
the multi-dimensional and nonlinear setting to develop predictive and
efficient iterative solution methods for more general force-based
hybrid and multiphysics
methods~\cite{hybrid_review,kohlhoff,cadd,Miller:2008}.
Our investigations may also prove relevant for some hybrid methods that
utilize overlapping or bridging domains
~\cite[see Method III]{BadiaParksBochevGunzburgerLehoucq:2007}.

\appendix
\section{Eigenbasis Computation for $\Li^{-1} \Lqcf$.}
\label{app:cond_calc}
We note that care must be taken when computing the basis of
eigenvectors since the eigenvalue $A_F$ has a multiplicity of
$(2N-2K-2)$ (cf. Lemma \ref{th:spec_Lqnl_U12}). This renders the
problem highly ill-conditioned and naive usage of a standard
eigensolver leads to unstable results. To circumvent this difficulty,
we observe from \eqref{qcf} that $\Lqcf e_j=A_F\Li e_j$ for
$j=-N+1,\dots, -K-3$ and $j= K+3,\dots, N-1$, and hence
$\Li^{-1}\Lqcf$ has the block structure
\begin{equation*}
    \Li^{-1}\Lqcf = {\scriptsize \left(\begin{array}{rrr|r|rrr}
    A_F &      &     & \multirow{3}{*}{${X_1}$} &     &      &   \\[-1mm]
        &\ddots&     &      &   &      &   \\[-1mm]
        &      & A_F &      &   &      &   \\
      \hline
        &      &     & \multirow{3}{*}{${X_2}$} &     &      &   \\
	&      &     &   &   &      &   \\
	&      &     &   &   &      &   \\
      \hline
        &      &     & \multirow{3}{*}{${X_3}$} & A_F &      &   \\[-1mm]
        &      &     &     &     &\ddots&   \\[-1mm]
        &      &     &     &     &      & A_F
    \end{array} \right), }
\end{equation*}
where $X_2$ is a $(2K+5) \times (2K+5)$ matrix.  From this form, we
see that there are $2N-2K-6$ standard unit vectors that are
eigenvectors corresponding to the eigenvalue $A_F.$ According to
Lemma~\ref{th:spec_Lqnl_U12}, the multiplicity of $A_F$ is $2N-2K-2,$
so that we have accounted for all but four eigenvectors of the high
multiplicity eigenvalue $A_F.$

Next, we reduce the dimensionality of the eigenvalue problem to
\begin{equation*}
  X_2 v_2 = \lambda v_2.
\end{equation*}
We then extend these eigenvectors to eigenvectors of $\Li^{-1} \Lqcf$
by defining
\begin{equation*}
  v = \left[
    \begin{array}{c}
      v_1\\
      v_2\\
      v_3
    \end{array}
  \right]
\end{equation*}
where $v_1:=(\lambda-A_F)^{-1}X_1 v_2$ and $v_3:=(\lambda-A_F)^{-1}X_3
v_2.$ Note that $v_i$ ($i=1,3$) is well defined provided that $\lambda
\neq A_F$ or $X_i v_2 = 0,$ and we observe numerically that $X_i v_2 =
0$ whenever $\lambda = A_F.$ Finally, the eigenvectors obtained in
this manner are normalized before computing the condition number of
the eigenbasis.

\end{document}